\documentclass[11pt, a4paper]{article}

\usepackage{amsmath}
\usepackage{amsfonts}
\usepackage{amssymb}
\usepackage[english]{babel}
\usepackage{xcolor}
\usepackage{euscript}
\usepackage{eufrak}

\def\english{\selectlanguage{english}}

\providecommand\mathbb{\bf}
\newcommand\R{{\mathbb R}}
\newcommand\N{{\mathbb N}}

\newcommand\Sp{{\mathbb S}}

\newtheorem{thm}{Theorem}[section]
\newtheorem{lemma}{Lemma}[section]
\newtheorem{pro}{Proposition}[section]

\newtheorem{remark}{Remark}[section]

\newcounter{Remark}

\renewcommand\theRemark{\arabic{Remark}}

\newcounter{steps}
\newenvironment{proof}[1][]{%
\par\medbreak\setcounter{steps}{0}
{\noindent\bfseries Proof#1. }} {\hfill\fbox{\ }\medbreak}

\newcounter{substeps}[steps]

\newcommand{\intv}[1]{
\int _{\R ^d} \!#1 \;\mathrm{d}v}

\newcommand{\intw}[1]{
\int _{\R ^d} \!#1 \;\mathrm{d}w}

\newcommand{\intvp}[1]{
\int _{\R ^d} \!#1 \;\mathrm{d}v^\prime}

\newcommand{\vp}[0]{
v^{\prime}}

\newcommand{\fe}[0]{
f ^\varepsilon}

\newcommand{\Divx}[0]{
\mathrm{div}_x}

\newcommand{\Divv}[0]{
\mathrm{div}_v}

\newcommand{\fz}[0]{
f}

\newcommand{\fo}[0]{
f ^{1}}

\newcommand{\lime}[0]{
\lim _{\varepsilon \searrow 0}}

\newcommand{\muf}[0]{
M_{u [\fz]}}

\newcommand{\calmu}[0]{
{\mathcal M}_{u}}

\newcommand{\Phiu}{
\Phi _u}

\newcommand{\sphere}[0]{
\mathbb{S}^{d-1}}

\newcommand{\ltmu}{
L^2 _{M_u}}

\newcommand{\bltmu}{
{\bf L}^2 _{M_u}}

\newcommand{\homu}{
H^1 _{M_u}}

\newcommand{\bhomu}{
{\bf H}^1 _{M_u}}

\newcommand{\thomu}{
\tilde{H}^1 _{M_u}}

\newcommand{\bthomu}{
{\bf \tilde{H}}^1 _{M_u}}

\newcommand{\intth}[1]{
\int _0 ^{\pi} #1 \;\mathrm{d}\theta}

\newcommand{\calL}[0]{
\mathcal{L}}

\newcommand{\sumi}[0]{
\sum _{i = 1}^{d-1}}

\newcommand{\sumj}[0]{
\sum _{j = 1}^{d-1}}

\newcommand{\eps}[0]{
\varepsilon}

\newcommand{\calTu}[0]{
\mathcal{T}_u}

\newcommand{\calTz}[0]{
\mathcal{T}_0}

\newcommand{\calO}[0]{
\mathcal{O}}

\newcommand{\vb}[0]{
\overline{v}}

\newcommand{\Exp}[1]{
\exp \left (#1 \right )}

\newcommand{\md}[0]{
\mathrm{d}}

\newcommand{\intcr}[2]{
\int _{\R_+}#1\int _{-1} ^{+1} #2 \;\mathrm{d}c\mathrm{d}r}

\newcommand{\vab}{
V_{\alpha, \beta} (|v|)}

\newcommand{\lims}[0]{
\lim _{\sigma \searrow 0}}

\begin{document}
\english

\title{Fluid models with phase transition for kinetic equations in swarming}

\author{Miha\"i BOSTAN \thanks{Aix Marseille Universit\'e, CNRS, Centrale Marseille, Institut de Math\'ematiques de Marseille, UMR 7373, Ch\^ateau Gombert 39 rue F. Joliot Curie, 13453 Marseille FRANCE. E-mail : {\tt mihai.bostan@univ-amu.fr}}
\;,\; Jos\'e Antonio CARRILLO
\thanks{Department of Mathematics, Imperial College London, London SW7 2AZ UK.  E-mail : {\tt carrillo@imperial.ac.uk}}
}

\date{ (\today)}

\maketitle

\begin{abstract}
We concentrate on kinetic models for swarming with individuals interacting through self-propelling and friction forces, alignment and noise. We assume that the velocity of each individual relaxes to the mean velocity. In our present case, the equilibria depend on the density and the orientation of the mean velocity, whereas the mean speed is not anymore a free parameter and a phase transition occurs in the homogeneous kinetic equation. We analyze the profile of equilibria for general potentials identifying a family of potentials leading to phase transitions. Finally, we derive the fluid equations when the interaction frequency becomes very large.
\end{abstract}


\\
\\

\section{Introduction}
\label{intro}
This paper concerns the derivation of fluid models for populations of self-propelled individuals, with alignment and noise \cite{CarDorPan09, UAB25, ChuDorMarBerCha07} starting from their kinetic description. The alignment between particles is imposed by relaxing the individuals velocities towards the mean velocity \cite{CFRT10, review, CS2, HL08, HT08, MT11}. We refer to \cite{Neu77,BraHep77,Dob79,CCR10,BCC11,BCC12,HLL09,CCH,CCHS} and the references therein for a derivation of kinetic equations for collective behavior from microscopic models. 

We concentrate on models with phase transition \cite{BD14, BCCD16, Li19, DFL10, DFL15, FL11, VicCziBenCohSho95}. We denote by $f = f(t,x,v)\geq 0$ the particle density in the phase space $(x,v) \in \R^d \times \R^d$, with $d \geq 2$. The self-propulsion and friction mechanism writes $\Divv \{f \nabla _v V(|\cdot|)\}$, where $v \mapsto V(|v|)$ is a confining potential. When considering $V_{\alpha, \beta} (|v|) = \beta \frac{|v|^4}{4} - \alpha \frac{|v|^2}{2}$, with $\alpha, \beta >0$, we obtain the term $\Divv \{f(\beta |v|^2 - \alpha ) v \}$ see \cite{BosCar13, BosCar15} and also \cite{BosAsyAna, BosTraEquSin, BosGuiCen3D} for results based on averaging methods in magnetic confinement. The relaxation towards the mean velocity is given by $\Divv \{f (v - u[f]) \}$ cf. \cite{DM08}, where for any particle density the notation $u[f]$ stands for the mean velocity
\[
u[f] = \frac{\intv{f(v)\;v}}{\intv{f(v)}}.
\]
Including noise with respect to the velocity variable, we obtain the Fokker-Planck type equation
\begin{equation}
\label{Equ1}
\partial _t f + v \cdot \nabla _x f = Q(f):= \Divv\{\sigma \nabla _v f + f(v - u[f]) + f \nabla _v V(|\cdot|)\},\;\;(t,x,v) \in \R_+ \times \R^d \times \R^d\,.
\end{equation}
When considering large time and space scales in \eqref{Equ1}, we are led to the kinetic equation
\begin{equation}
\label{Equ3}
\partial _t \fe + v \cdot \nabla _x \fe = \frac{1}{\eps} Q(\fe),\;\;(t,x,v) \in \R_+ \times \R^d \times \R^d.
\end{equation}
We investigate the asymptotic behavior of the family $(\fe)_{\eps >0}$, when $\eps$ becomes small. We expect that the limit density $f(t,x,\cdot) = \lime \fe (t,x,\cdot)$ is an equilibrium for the interaction mechanism
\[
Q(f(t,x,\cdot)) = 0,\;\;(t,x) \in \R_+ \times \R^d. 
\]
For any $u \in \R^d$ we introduce the notations
\[
\Phiu (v) = \frac{|v - u|^2}{2} + V(|v|),\;\;Z(\sigma, u) = \intv{\exp \left ( - \frac{\Phiu (v)}{\sigma} \right ) },\;\;M_u (v) = \frac{\exp \left ( - \frac{\Phiu (v)}{\sigma} \right ) }{Z(\sigma, u)}.
\]
Actually the function $Z$ depends only on $\sigma$ and $|u|$, see Proposition \ref{Current}, and thus we will write $Z = Z(\sigma, l = |u|)$. Notice that for any smooth particle density $f$ and any $u \in \R^d$ we have
\begin{equation*}
\sigma \nabla _v f + f(v - u[f]) + f \nabla _v V(|\cdot|) = \sigma M_u (v) \nabla _v \left ( \frac{f}{M_u}\right )
\end{equation*}
leading to the following representation formula
\[
Q(f) = \sigma \Divv \left ( \muf \nabla _v \left ( \frac{f}{\muf }\right )\right ).
\]
Multiplying by $f/\muf$ and integrating by parts with respect to the velocity imply that any equilibrium satisfies
\[
f = \rho [f] \muf,\;\;\rho[f] = \intv{f(v)}.
\]
Recall that $u[f]$ is the mean velocity, and therefore we impose
\begin{equation}
\label{Equ6}
\intv{f(t,x,v) (v - u[f(t,x,\cdot)])} = 0,\;\;(t,x) \in \R_+ \times \R^d.
\end{equation}
Notice that $\Phiu$ is left invariant by any orthogonal transformation preserving $u$. Consequently, we deduce (see Proposition \ref{Current}) that $\intv{f(v)\;v}$ is parallel to $u$, and therefore the constraint \eqref{Equ6} fix only the modulus of the mean velocity, and not its orientation (which remains a free parameter). 

Our first important observation gives a characterization to find the bifurcation diagram of stationary solutions of $Q(f) = 0$. We prove that $M_u$ is an equilibrium if and only if $l = |u|$ is a critical point of $Z(\sigma, \cdot)$, cf. Proposition \ref{Current}. Moreover, several values for $|u|$, or only one are admissible, depending on the diffusion coefficient $\sigma$. In that case we will say that a phase transition occurs. Notice that in this work we do not distinguish between phase transitions and bifurcation points. For any particle density $f = f(v)$, the notation $\Omega[f]$ stands for the orientation of the mean velocity $u[f]$, if $u[f] \neq 0$
\[
\Omega[f] = \frac{u[f]}{|u[f]|} = \frac{\intv{f(v) \;v }}{\left | \intv{f(v) \;v }  \right |}
\]
and any vector in $\sphere$, if $u[f] = 0$. Here $\sphere$ is the set of unit vectors in $\R^d$. Notice also that we always have
\[
u[f] = |u[f]|\; \Omega [f].
\]
Finally, for any $(t,x) \in \R_+ \times \R^d$, the limit particle density is a von Mises-Fisher distribution $f(t,x,v) = \rho (t,x) M_{|u|\Omega(t,x)} (v)$ parametrized by the concentration $\rho (t,x) = \rho [f(t,x,\cdot)]$ and the orientation $\Omega (t,x) = \Omega [f(t,x,\cdot)]$. We identify a class of potentials $v \mapsto V(|v|)$ such that a phase transition occurs and we derive the fluid equations satisfied by the macroscopic quantities $\rho, \Omega$. More exactly we assume that the potential $v \mapsto V(|v|)$ satisfies
\begin{equation}
\label{EquWellDefZ} 
\lim _{|v| \to + \infty } \frac{\frac{|v|^2}{2} + V(|v|) }{|v|} = +\infty
\end{equation}
(such that $Z$ is well defined) and belongs to the family $\mathcal{V}$ defined by: there exists $\sigma _0>0$ verifying
\begin{enumerate}
\item
For any $0 < \sigma < \sigma _0$ there is $l(\sigma) >0 $ such that $Z(\sigma,l)$ is stricly increasing on $[0,l(\sigma)]$ and strictly decreasing on $[l(\sigma), +\infty[$;
\item
For any $\sigma \geq \sigma _0, Z(\sigma,l)$ is strictly decreasing on $[0,+\infty[$. 
\end{enumerate}
The first important result in this work shows that potentials in $\mathcal{V}$ have a phase transition at $\sigma=\sigma_0$ as shown in Section 2.

\begin{remark}
The potential $V(|v|) = \beta \frac{|v|^4}{4} - \alpha \frac{|v|^2}{2}$ belongs to the family $\mathcal{V}$ as shown in \cite{Tugaut2013Phase,BCCD16,Li19} in any dimension.
\end{remark}

\begin{thm}
\label{MainRes1}
Assume that the potential $v \mapsto V(|v|)$ satisfies \eqref{EquWellDefZ}, belongs to the family $\mathcal{V}$ defined above and that $0 < \sigma < \sigma _0$. Let us consider $(\fe)_{\eps >0}$ satisfying
\begin{equation}\label{asymplim}
\partial _t \fe + v \cdot \nabla _x \fe = \frac{1}{\eps} \Divv \{\sigma \nabla _v \fe + \fe ( v - u [\fe] + \nabla _v V(|\cdot|)\;) \},\;(t,x,v) \in \R_+ \times \R^d \times \R^d. 
\end{equation}
Therefore, at any $(t,x) \in \R_+ \times \R^d$ the dominant term in the Hilbert expansion $\fe = f + \eps \fo + ...$ is an equilibrium distribution of $Q$, that is $f(t,x,v) = \rho (t,x) M_{u(t,x)} (v)$, where 
\begin{equation}
\label{Equ50} 
u(t,x) = l(\sigma) \Omega (t,x),\;\;(t,x) \in \R_+ \times \R^d
\end{equation}
\begin{equation}
\label{Equ51}
\partial _t \rho + \Divx (\rho u ) = 0,\;\;(t,x) \in \R_+ \times \R^d
\end{equation}
\begin{equation}
\label{Equ52}
\partial _t \Omega + l(\sigma) c_{\perp} \;(\Omega \cdot \nabla _x ) \Omega + \frac{\sigma}{l(\sigma)} (I_d - \Omega \otimes \Omega ) \frac{\nabla _x \rho}{\rho} = 0,\;\;(t,x) \in \R_+ \times \R^d. 
\end{equation}
The constant $c_\perp$ is given by
\[
c_\perp = \frac{\int_{\R_+}r^{d+1} \intth{\cos \theta \, \chi (\cos \theta, r) \,  e(\cos \theta, r, l(\sigma)) \sin ^{d-1} \theta}\md r}{l(\sigma)\int_{\R_+}r^{d} \intth{\chi (\cos \theta, r) \, e(\cos \theta, r, l(\sigma)) \sin ^{d-1} \theta}\md r}
\]
and the function $\chi$ solves
\begin{align}
\label{EquChiRes1}
- \sigma & \partial _c \left [r^{d-3} (1 - c^2) ^{\frac{d-1}{2}} e(c,r, l(\sigma)) \partial _c \chi   \right ]  - \sigma \partial _r \left [r^{d-1}(1 - c^2 ) ^{\frac{d-3}{2}} e(c,r, l(\sigma)) \partial _r \chi  \right ] \\
& + \sigma (d-2) r ^{d-3} (1- c^2) ^{\frac{d-5}{2}} e(c,r, l(\sigma)) \chi = r^d (1-c^2) ^{\frac{d-2}{2}}e(c,r,l(\sigma)), \; (c,r) \in ]-1,1[ \times \R_+\nonumber 
\end{align}
where $e(c, r, l) = \exp \left ( - \frac{r^2}{2\sigma} + \frac{rcl}{\sigma} - \frac{V(r)}{\sigma} \right )$. 
\end{thm}

\begin{remark}
Several considerations regarding the hydrodynamic equations \eqref{Equ50}-\eqref{Equ52} and the asymptotic limit to obtain them are needed:
\begin{itemize}

\item The asymptotic limit in \eqref{asymplim} is different from the one analysed in \cite{BosCar15} where the friction term is penalized at higher order. The main technical difficulty in \cite{BosCar15} compared to our present work is that to solve for the different orders on the expansion in \cite{BosCar15} we had to deal with Fokker-Planck equations on the velocity sphere with speed $\sqrt{\frac{\alpha}{\beta}}$.

\item The hydrodynamic equations  \eqref{Equ50}-\eqref{Equ52} in the particular case of the potential $V(|v|) = \beta \frac{|v|^4}{4} - \alpha \frac{|v|^2}{2}$ recover the ones obtained in \cite{DM08,DFL10,DFL15,BosCar15} by taking the limit $\alpha\to\infty$ with $\beta/\alpha=O(1)$. In this limit, the particle density $f$ is squeezed to a Dirac on the velocity sphere with speed  $\sqrt{\frac{\alpha}{\beta}}$. The constants can be computed exactly based on \cite{Li19} and they converge towards the exact constants obtained in \cite{DM08,DFL15,BosCar15}. This is left to the reader for verification.

\item The hydrodynamic equations  \eqref{Equ50}-\eqref{Equ52} have the same structure as the equations derived in \cite{DM08,DFL15,BosCar15} just with different constants, and therefore they form a hyperbolic system as shown in \cite[Subsection 4.4]{DM08}.

\end{itemize}
\end{remark}

When $V(|\cdot|)$ belongs to the family $\mathcal{V}$, we know that $|u| \in \{0,l(\sigma)\}$, for any $0 < \sigma < \sigma _0$ and $|u| = 0$ for any $\sigma \geq \sigma _0$. There is no time evolution for $|u|$. But the modulus of the mean velocity evolves in time for other potentials. For example, let us assume that there is $\sigma >0$, $ 0 \leq l_1 (\sigma) < l_2 (\sigma)\leq +\infty$ such that the function $l \mapsto Z(\sigma, l)$ is stricly increasing on $[0,l_1(\sigma)]$, constant on $[l_1(\sigma), l_2(\sigma)[$, and strictly decreasing on $[l_2(\sigma), +\infty[$. In that case, we obtain a balance for $|u|$ as well. 
\begin{thm}
\label{MainRes2}
Assume that the potential $v \mapsto V(|v|)$ satisfies \eqref{EquWellDefZ} and verifies the above hypothesis for some $\sigma >0$. Let us consider $(\fe)_{\eps >0}$ satisfying
\[
\partial _t \fe + v \cdot \nabla _x \fe = \frac{1}{\eps} \Divv \{\sigma \nabla _v \fe + \fe ( v - u [\fe] + \nabla _v V(|\cdot|) \; )\},\;\;(t,x,v) \in \R_+ \times \R^d \times \R^d. 
\]
Therefore, at any $(t,x) \in \R_+ \times \R^d$ the dominant term in the Hilbert expansion $\fe = f + \eps \fo + ...$ is an equilibrium distribution of $Q$, that is $f(t,x,v) = \rho (t,x) M_{u(t,x)} (v)$, where 
\begin{equation}
\label{Equ57}
\partial _t \rho + \Divx (\rho u ) = 0,\;\;(t,x) \in \R_+ \times \R^d
\end{equation}
\begin{align}
\label{Equ58}
\partial _t u & + [ c_{\perp} (I_d - \Omega \otimes \Omega) + c_\parallel \Omega \otimes \Omega] (u\cdot \partial _x) u +  [ (c _\perp - 1) (I_d - \Omega \otimes \Omega) + (c_\parallel - 1) \Omega \otimes \Omega ] \nabla _x \frac{|u|^2}{2} \nonumber \\
& + \sigma \frac{\nabla _x \rho }{\rho} + c_\parallel ^\prime \;\Divx \Omega \;|u| u = 0,\;\;(t,x) \in \R_+ \times \R^d. 
\end{align}
The constants $c_\perp, c_\parallel, c_\parallel ^\prime$ are given by
\[
c_\perp = \frac{\int_{\R_+}r^{d+1} \intth{\cos \theta \, \chi (\cos \theta, r) \,  e(\cos \theta, r, |u|) \sin ^{d-1} \theta}\md r}{|u|\int_{\R_+}r^{d} \intth{\chi (\cos \theta, r) \, e(\cos \theta, r, |u|) \sin ^{d-1} \theta}\md r}
\]
\[
c_\parallel = \frac{\int_{\R_+}r^{d+1} \intth{\cos ^2 \theta \, \chi _\Omega (\cos \theta, r) \,  e(\cos \theta, r, |u|) \sin ^{d-2} \theta}\md r}{2 |u|\int_{\R_+}r^{d} \intth{\cos \theta \chi _\Omega (\cos \theta, r) \, e(\cos \theta, r, |u|) \sin ^{d-2} \theta}\md r}
\]
\[
c_\parallel ^\prime = \frac{\int_{\R_+}r^{d+1} \intth{\chi _\Omega (\cos \theta, r) \,  e(\cos \theta, r, |u|) \sin ^{d} \theta}\md r}{(d-1) |u|\int_{\R_+}r^{d} \intth{\cos \theta \chi _\Omega (\cos \theta, r) \, e(\cos \theta, r, |u|) \sin ^{d-2} \theta}\md r}
\]
the function $\chi$ solves \eqref{EquChiRes1} and the function $\chi _\Omega$ solves
\begin{align*}
- & \sigma \partial _c \{ r ^{d-3} (1 - c^2) ^{\frac{d-1}{2}} e(c,r,|u|) \partial _c \chi _\Omega \} - \sigma \partial _r \{ r ^{d-1} (1 - c^2 ) ^{\frac{d-3}{2}} e(c, r, |u|) \partial _r \chi _\Omega \} \\
& = r^{d-1 } (r c - |u|)  (1 - c^2) ^{\frac{d-3}{2}}e(c,r,|u|),\;(c,r) \in  ]-1,1[ \times ]0,+\infty[\nonumber.
\end{align*}
\end{thm}
Our paper is organized as follows. In Section \ref{PropEqui} we investigate the function $Z$, whose variations will play a crucial role when determining the equilibria of the interaction mechanism $Q$. We identify a family of potentials such that a phase transition occurs for some critical diffusion coefficient $\sigma _0$. Section \ref{LinIntMec} is devoted to the study of the linearization of $Q$ and of its formal adjoint. We are led to study the spectral properties of the pressure tensor. The kernel of the adjoint of the linearization of $Q$ is studied in Section \ref{Invariants}. These elements will play the role of the collision invariants, when determining the macroscopic equations by the moment method. The main results, Theorem \ref{MainRes1}, \ref{MainRes2}, are proved in Section \ref{MainTh}. Some examples are presented in Section \ref{Example}.


\section{Phase transitions and Potentials: Properties of Equilibria}
\label{PropEqui}
For any $u \in \R^d$ we denote by $\calTu$ the family of orthogonal transformations of $\R^d$ preserving $u$. Notice that $\calTz$ is the family of all orthogonal transformations of $\R^d$. 
\begin{remark}
\label{Sym} The functions on $\R^d$ which are left invariant by the family $\calTz$ are those depending only on $|v|$. The functions on $\R^d$ which are left invariant by the family $\calTu, u \neq 0$, are those depending on $v \cdot u$ and $|v|$.
\end{remark}
\begin{lemma}
\label{InvVectField}
Let $u$ be a vector in $\R^d$ and $a : \R^d \to \R^d$ be a integrable vector field on $\R^d$, which is left invariant by the family $\calTu$ {\it i.e.,}
\[
a(^t \mathcal{O}v ) = \;^t \mathcal{O} a(v),\;\;v \in \R^d,\;\;\mathcal{O} \in \calTu.
\]
Then $\intv{a(v)} \in \R u$.
\end{lemma}
\begin{proof}
For any $\calO \in \calTu$, we have
\[
\intv{a(v)} = \intvp{a(\;^t \calO \vp ) } = \;^t \calO \intvp{a(\vp)}.
\]
For any $\xi \in \sphere \cap ( \R u ) ^\perp$, we consider $\calO _\xi  = I_d - 2 \xi \otimes \xi \in \calTu$, and thus we obtain
\[
\intv{a(v)} = (I_d - 2 \xi \otimes \xi ) \intvp{a(\vp)},
\]
or equivalently $\xi \cdot \intv{a(v)} = 0$. Therefore, we have $\intv{a(v)} \in ((\R u ) ^\perp) ^ \perp = \R u$.
\end{proof}
We assume that 
\begin{equation}
\label{Equ11}
\lim _{|v| \to + \infty } \frac{\frac{|v|^2}{2} + V(|v|) }{|v|} = +\infty.
\end{equation}
Observe that 
\begin{align*}
Z(\sigma, u ) & = \exp \left ( - \frac{|u|^2}{2\sigma} \right ) \intv{ \exp \left ( - \frac{\frac{|v|^2}{2} + V(|v|) }{\sigma} + \frac{v \cdot u }{\sigma} \right ) } \\
& \leq \exp \left ( - \frac{|u|^2}{2\sigma} \right )\intv{ \exp \left [ - \frac{|v|}{\sigma} \left ( \frac{\frac{|v|^2}{2} + V(|v|) }{|v|} -  |u|\right ) \right ] },
\end{align*}
and therefore, under the hypothesis \eqref{Equ11}, it is easily seen that $Z(\sigma, u)$ is finite for any $\sigma >0$ and $u \in \R^d$. Similarly we check that for any $\sigma >0$ and $u \in \R^d$, all the moments of $M_u$ are finite
\[
\intv{|v|^p M_u (v)} < +\infty, \;\;p \in \N.
\]
For further developments, we recall the formula
\begin{equation}
\label{Equ12}
\intv{\chi \left ( \frac{v \cdot \Omega}{|v|},|v| \right )} = |\Sp ^{d-2}| \int _{\R_+}  r^{d-1} \intth{\chi(\cos \theta, r)  \sin ^{d-2} \theta }\mathrm{d}r,
\end{equation}
for any non negative measurable function $\chi = \chi (c,r) : ]-1,1[\times \R_+ ^\star \to \R$, any $\Omega \in \Sp ^{d-1}$ and $d \geq 2$. Here $|\Sp ^{d-2}|$ is the surface of the unit sphere in $\R^{d-1}$, for $d \geq 3$, and $|\Sp ^0| = 2$ for $d = 2$. 
\begin{pro}
\label{Current}
Assume that the potential $v \mapsto V(|v|)$ satisfies \eqref{Equ11}. Then the following statements hold true~:
\begin{enumerate}
\item
The function $Z(\sigma, u)$ depends only on $\sigma$ and $|u|$. We will simply write
\[
\intv{\exp \left ( - \frac{\Phiu (v)}{\sigma} \right ) }= Z(\sigma, l = |u|).
\]
\item
For any $u \in \R^d$, we have $\intv{M_u (v) v } \in \R_+ u$ and obviously, $\intv{M_0 (v) v } = 0$.

\item
The von Mises-Fisher distribution $M_u$ is an equilibrium if and only if $\partial _l Z(\sigma, l) = 0$. For any $\sigma >0, M_0 (v) = Z^{-1} (\sigma, 0) \exp \left ( - \Phi _0 (v)/\sigma\right )$ is an equilibrium. 
\end{enumerate}
\end{pro}
\begin{proof}$\;$\\
1. 
Applying formula \eqref{Equ12} with $\Omega = u /|u|$, if $u \neq 0$, and any $\Omega \in \Sp ^{d-1}$ if $u = 0$, we obtain
\begin{align*}
Z & = \intv{\exp \left ( - \frac{|v|^2}{2\sigma} - \frac{|u|^2}{2\sigma} + \frac{v\cdot u}{\sigma} - \frac{V(|v|)}{\sigma} \right )}\\
& = |\Sp ^{d-2} |  \exp \left ( - \frac{|u|^2}{2\sigma}\right ) \int_{\R_+} \exp \left ( - \frac{r^2}{2\sigma} - \frac{V(r)}{\sigma} \right ) r ^{d-1} \intth{\exp \left ( \frac{r |u| \cos \theta }{\sigma} \right ) \sin ^{d-2} \theta }\mathrm{d}r ,
\end{align*}
and therefore $Z$ depends only on $\sigma$ and $|u|$. \\
2. We consider the integrable vector field $a(v) = M_u (v) v, v \in \R^d$. It is easily seen that for any $\calO \in \calTu$, we have
\[
\Phiu (\;^t \calO v) = \Phiu (v),\;\;M_u (\;^t \calO v ) = M_u (v),\;\;v \in \R^d,
\]
and therefore the vector field $a$ is left invariant by $\calTu$. Our conclusion follows by Lemma \ref{InvVectField}. It remains to check that $\intv{M_u (v) ( v \cdot u) } >0$, when $u \neq 0$. Indeed, we have
\[
Z\intv{M_u (v) ( v \cdot u ) } = \int _{v \cdot u >0} \left [ \exp \left (- \frac{\Phiu (v)}{\sigma}\right ) - \exp \left (- \frac{\Phiu (-v)}{\sigma}\right )  \right ] ( v \cdot u ) \;\mathrm{d}v,
\]
and we are done observing that for any $v$ such that $v \cdot u >0$ we have
\[
- \Phiu (v) = - \frac{|v|^2}{2} + v \cdot u - \frac{|u|^2}{2} - V(|v|) > - \frac{|v|^2}{2} - v \cdot u - \frac{|u|^2}{2} - V(|v|) = - \Phiu (-v).
\]
3. The von Mises-Fisher distribution $M_u$ is an equilibrium if and only if $\intv{M_u (v) (v-u) } = 0$. By the previous statement we know that $\intv{M_u (v) v } \in \R u$ and therefore $M_u$ is an equilibrium iff $\intv{M_u (v) (v \cdot \Omega - |u|)} = 0$, where $\Omega = \frac{u}{|u|}$ if $u \neq 0$ and $\Omega$ is any vector in $\Sp ^{d-1}$ if $u = 0$. But we have
\begin{align}
\label{Equ12Bis}
\partial _l Z(\sigma, |u|) & = |\Sp ^{d-2}| \exp \left ( - \frac{|u|^2}{2\sigma}\right ) \int _{\R_+} \exp \left ( - \frac{r^2}{2\sigma} - \frac{V(r)}{\sigma} \right ) r^{d-1} \\
&\,\,\,\,\,\, \times \intth{\exp \left ( \frac{r |u| \cos \theta}{\sigma} \right ) \frac{r \cos \theta - |u|}{\sigma} \sin ^{d-2} \theta }\mathrm{d}r \nonumber \\
& = \intv{\exp \left ( - \frac{\Phiu (v)}{\sigma} \right ) \frac{v \cdot \Omega - |u|}{\sigma} } \nonumber \\
& = \frac{Z(\sigma, |u|)}{\sigma}\intv{M_u (v) ( v \cdot \Omega - |u| ) },\nonumber 
\end{align}
and therefore $M_u$ is an equilibrium if and only if $l = |u|$ is a critical point of $Z(\sigma, \cdot)$. 
\end{proof}
\begin{remark}
\label{Modulus}
As $Z$ depends only on $\sigma, |u|$, we can write
\begin{align*}
Z(\sigma, |u|) & = \intv{\exp \left ( - \frac{|v - \Omega |u||^2}{2\sigma} - \frac{V(|v|)}{\sigma} \right )}\\
& = \intv{\exp \left (- \frac{|v|^2}{2\sigma} + \frac{(v \cdot \Omega) |u|}{\sigma} - \frac{|u|^2}{2\sigma} - \frac{V(|v|)}{\sigma}   \right ) },
\end{align*}
for any $\Omega \in \Sp ^{d-1}$ and $u \in \R^d$. We deduce that for any $\Omega \in \Sp ^{d-1}$ and $u \in \R^d$, we have
\begin{align*}
\partial _l Z(\sigma, |u|) & =  \intv{\exp \left ( - \frac{|v|^2}{2\sigma} + \frac{(v \cdot \Omega) |u|}{\sigma} - \frac{|u|^2}{2\sigma} - \frac{V(|v|)}{\sigma}\right )\frac{v \cdot \Omega - |u|}{\sigma}}\\
& = \intv{\exp \left ( - \frac{\Phi_{|u|\Omega} (v)}{\sigma} \right )  \frac{v \cdot \Omega - |u|}{\sigma} }\\
& = \intv{\exp \left ( - \frac{\Phiu (v)}{\sigma} \right ) \frac{(v-u ) \cdot \Omega [u]}{\sigma}}
\end{align*}
and
\begin{align*}
\partial ^2 _{ll} Z(\sigma, |u|) & =  \intv{\exp \left ( - \frac{\Phi_{|u|\Omega} (v)}{\sigma} \right )  \frac{[v \cdot \Omega - |u|\;]^2 - \sigma}{\sigma ^2} }\\
& = \intv{\exp \left ( - \frac{\Phiu (v)}{\sigma} \right ) \frac{[(v-u ) \cdot \Omega [u]\;]^2 - \sigma}{\sigma ^2}},
\end{align*}
where $\Omega = \frac{u}{|u|}$ if $u \neq 0$ and $\Omega$ is any vector in $\Sp ^{d-1}$ if $u = 0$ (compare with \eqref{Equ12Bis}, established for $\Omega = u/|u|$, if $u \neq 0$).
\end{remark}
At this point, we know that for any $\sigma >0$, the equilibria are related to the critical points of $Z(\sigma, \cdot)$. In order to find possible bifurcation points of the disordered state $u=0$, let us analyze the variations of $Z(\sigma, \cdot)$ for small $\sigma$. We assume the following hypothesis on the potential
\begin{equation}
\label{Equ14}
V ( |\cdot|) \in C^2 (\R^d),\;\;v \mapsto \frac{|v|^2}{2} + V(|v|)\; \mbox{ is strictly convex on } \R^d.
\end{equation}
For such a potential, we can minimize $\Phiu(v)$ with respect to $v \in \R^d$, for any $u \in \R^d$. Indeed, the function $\Phiu$ is convex, continuous on $\R^d$ and
\begin{align*}
\Phiu (v) & = \frac{|v- u |^2}{2} + V(|v|) = \frac{|v|^2}{2} + V(|v|) - v \cdot u + \frac{|u|^2}{2} \nonumber \\
& = |v| \left (\frac{\frac{|v|^2}{2} + V (|v|)}{|v|} - \frac{v \cdot u }{|v|}    \right ) + \frac{|u|^2}{2} \\
& \geq  |v| \left (\frac{\frac{|v|^2}{2} + V (|v|)}{|v|} - |u|   \right ) + \frac{|u|^2}{2}.
\end{align*}
By \eqref{Equ11} we deduce that $\lim _{|v| \to +\infty} \Phiu (v) = +\infty$ and therefore $\Phiu $ has a minimum point $\vb \in \R^d$. This minimum point is unique (use $\vb - u + (\nabla _v V ( |\cdot |) ) (\vb) = 0$ and the strict convexity of $v \mapsto \frac{|v|^2}{2} + V(|v|)$\;). We intend to analyze the sign of $\partial _l Z(\sigma, |u|)$ for small $\sigma$. Performing the change of variable $v = \vb + \sqrt{\sigma} w$ leads to 
\begin{align}
\label{Equ17}
& \partial _l Z (\sigma, |u|)  \sigma ^{1 - d/2} \exp \left ( \frac{\Phiu (\overline{v})}{\sigma} \right )  = \intv{\Exp{- \frac{\Phiu (v) - \Phiu (\overline{v}) - \nabla _v \Phiu (\overline{v}) \cdot (v - \overline{v})}{\sigma} }\\
& \times \frac{(v-u) \cdot \Omega [u] }{\sigma ^{d/2}}} \nonumber \\
&  = \int_{\R^d}\Exp{- \frac{\Phiu(\overline{v} + \sqrt{\sigma} w) - \Phiu(\overline{v}) - \sqrt{\sigma} \nabla _v \Phiu (\overline{v}) \cdot w}{\sigma}} ( \overline{v} + \sqrt{\sigma} w - u ) \cdot \Omega[u]\;\mathrm{d}w \nonumber \\
& = \int_{\R^d}\Exp{- \frac{\Phi _0(\overline{v} + \sqrt{\sigma} w) - \Phi _0(\overline{v}) - \sqrt{\sigma} \nabla _v \Phi _0 (\overline{v}) \cdot w}{\sigma}} ( \overline{v} + \sqrt{\sigma} w - u ) \cdot \Omega[u]\;\mathrm{d}w \nonumber.
\end{align}
We need to determine the sign of $(\overline{v} - u) \cdot \Omega[u]$, where $\overline{v}$ is the minimum point of $\Phiu$. As $V(|\cdot|) \in C^1(\R^d)$, we have $V^\prime (0) = 0$. We assume that $V(\cdot)$ possesses another critical point $r_0 >0$ and 
\begin{equation}
\label{Equ15} V^\prime (r) <0 \;\mbox{ for any } 0 < r < r_0 \;\mbox{ and } V^\prime (r) >0\;\mbox{ for any } r >r_0. 
\end{equation}
Notice that this is the case for $V_{\alpha, \beta} (r) = \beta \frac{r^4}{4} - \alpha \frac{r^2}{2}, \alpha, \beta >0$, with $r_0 = \sqrt{\alpha/\beta}$. 
\begin{pro}
\label{Sign}
Assume that \eqref{Equ11}, \eqref{Equ14}, \eqref{Equ15} hold true. Then
\begin{enumerate}
\item
The function $r \mapsto r + V^\prime (r)$ is strictly increasing on $\R_+$ and maps $[0,r_0]$ to $[0,r_0]$, and $]r_0, +\infty[$ to $]r_0, +\infty[$.

\item
We have
\[
(\overline{v} - u) \cdot \Omega[u] >0\;\mbox{ for any } 0 < |u| < r_0,\;\;\inf _{\delta \leq |u| \leq r_0 - \delta} (\overline{v} - u) \cdot \Omega[u] >0, \;\; 0 < \delta < \frac{r_0}{2}
\]
and
\[
(\overline{v} - u) \cdot \Omega[u] <0\;\mbox{ for any } |u| > r_0,\;\;\inf _{ |u| \geq r_0 + \delta} (u - \overline{v}) \cdot \Omega[u] >0, \;\; \delta >0.
\]
\end{enumerate}
\end{pro}
\begin{proof}$\;$\\
1. By \eqref{Equ14} we know that $\Phi _0$ is strictly convex on $\R^d$ and we deduce that $r \mapsto \frac{r^2}{2} + V(r)$ is strictly convex on $\R_+$. Therefore the function $r \mapsto r + V^\prime (r)$ is strictly increasing on $\R_+$ and maps $[0,r_0]$ to $[0,r_0]$. It remains to check that it is unbounded when $r \to +\infty$. Suppose that there is a constant $C$ such that $r + V^\prime (r) \leq C, r \in \R_+$. After integration with respect to $r$, one gets
\[
\frac{r^2}{2} + V(r) \leq V(0) + Cr,\;\;r \in \R_+,
\]
implying that 
\[
\frac{\frac{r^2}{2} + V(r)}{r} \leq \frac{V(0)}{r} + C,\;\;r \in \R_+,
\]
which contradicts \eqref{Equ11}.\\
2. Let us consider $0 < |u| < r_0$. Therefore, $\overline{v} \neq 0$ and
\[
\left ( |\overline{v}| + V^\prime (|\overline{v}| ) \right ) \frac{\overline{v}}{|\overline{v}|} = u ,
\]
implying that $ |\overline{v}| + V^\prime (|\overline{v}| ) = |u| \in ]0,r_0[$. By the previous statement we obtain $0 < |\overline{v}| < r_0$, $\Omega [\overline{v}] = \frac{\overline{v}}{|\overline{v}|} = \frac{u}{|u|} = \Omega [u]$, and thus
\[
( \overline{v} - u) \cdot \Omega[u] = - V^\prime(|\overline{v}|) \frac{\overline{v}}{|\overline{v}|} \cdot \Omega[u] = - V^\prime (|\overline{v}|)>0.
\]
Clearly, for any $0 < \delta < r_0/2$, we have
\[
\inf _{\delta \leq |u| \leq r_0 - \delta} (\overline{v} - u) \cdot \Omega[u] = \inf _{\delta \leq |u| \leq r_0 - \delta} ( - V^\prime(|\overline{v}|) ) >0.
\]
Similarly, for any $|u| > r_0$, we have $|\overline{v}| > r_0$ and
\[
( \overline{v} - u) \cdot \Omega[u] = - V^\prime(|\overline{v}|) \frac{\overline{v}}{|\overline{v}|} \cdot \Omega[u] = - V^\prime (|\overline{v}|)<0.
\]
As before, for any $\delta >0$, we obtain
\[
\inf _{ |u| \geq r_0 + \delta} (u - \overline{v} ) \cdot \Omega[u] = \inf _{|u| \geq r_0 + \delta}  V^\prime(|\overline{v}|)  >0.
\]
\end{proof}
The previous arguments allow us to complete the analysis of the variations of $Z(\sigma, |u|)$, when $\sigma$ is small. The convergence when $\sigma \searrow 0$ in \eqref{Equ17} can be handled by dominated convergence, provided that $w \mapsto |w| \Exp{- \frac{\partial _v ^2 \Phi _0 (\overline{v}) w \cdot w }{2}}$ belongs to $L^1(\R^d)$. We assume that there is $\lambda <1$ such that 
\begin{equation}
\label{Equ18}
v \mapsto V_\lambda (|v|) := \lambda \frac{|v|^2}{2} + V(|v|) \;\mbox{ is convex on } \R^d. 
\end{equation}
The potentials $V_{\alpha, \beta} (|v|) = \beta \frac{|v|^4}{4} - \alpha \frac{|v|^2}{2}, 0 < \alpha < 1, \beta >0$ satisfy the above hypothesis. Under \eqref{Equ18}, we write
\[
\Phi _0 (v) = (1 - \lambda ) \frac{|v|^2}{2} + V_\lambda (|v|),\;\;v \in \R^d,
\]
and therefore
\[
\partial _v ^2 \Phi _0 (v) = (1 - \lambda ) I_d + \partial _v ^2 V_\lambda (|\cdot|) \geq (1- \lambda ) I_d,\;\;v \in \R^d,
\]
implying that 
\[
\int_{\R^d} |w| \Exp{- \frac{\partial _v ^2 \Phi _0 (\overline{v}) w \cdot w }{2}} \;\mathrm{d}w \leq \int_{\R^d} |w| \Exp{- \frac{(1- \lambda) |w|^2 }{2}} \;\mathrm{d}w < + \infty.
\]
Notice that \eqref{Equ18} guarantees \eqref{Equ11} and \eqref{Equ14}. Indeed, the function $v \mapsto V_\lambda (|v|)$ being convex, it is bounded from below by a linear function
\[
\exists \; (v_\lambda, C_\lambda ) \in \R^d \times \R \;\mbox{ such that } V_\lambda (|v|) \geq ( v \cdot v_\lambda ) + C_\lambda,\;\;v \in \R^d,
\]
and therefore 
\[
\frac{\Phi _0 (v)}{|v|} = \frac{(1 - \lambda) \frac{|v|^2}{2} + V_\lambda (|v|)}{|v|} \geq (1 - \lambda) \frac{|v|}{2} - |v_\lambda | + \frac{C_\lambda}{|v|} \to +\infty,\;\mbox{ as } |v| \to +\infty. 
\]
Obviously, $\Phi _0 $ is strictly convex, as sum between the strictly convex function $v \mapsto (1- \lambda) \frac{|v|^2}{2}$ and the convex function $v \mapsto V_\lambda (|v|)$.

In order to conclude the study of the variations of $Z$ for small $\sigma >0$, we consider potentials $V$ satisfying $V(|\cdot|) \in C^2(\R^d)$, \eqref{Equ15} and \eqref{Equ18}. We come back to \eqref{Equ17}. Notice that 
\begin{align*}
\Phi _0 (\overline{v} + \sqrt{\sigma} w) - \Phi _0 (\overline{v}) - \sqrt{\sigma} \nabla _v \Phi _0 (\overline{v}) \cdot w & \geq (1- \lambda) \frac{|\overline{v} + \sqrt{\sigma}w|^2}{2} - (1 - \lambda) \frac{|\overline{v}|^2}{2} \\
& - (1- \lambda) \sqrt{\sigma} \overline{v} \cdot w =  (1 - \lambda) \sigma \frac{|w|^2}{2},
\end{align*}
implying that, for any $0 < \sigma \leq 1$
\begin{align*}
& \left  |\Exp{- \frac{\Phi _0 (\overline{v} + \sqrt{\sigma} w) - \Phi _0 (\overline{v}) - \sqrt{\sigma} \nabla _v \Phi _0 (\overline{v}) \cdot w }{\sigma} } (\overline{v} + \sqrt{\sigma} w - u) \cdot \Omega[u] \right | \\
& \leq \Exp{- (1 - \lambda)  \frac{|w|^2}{2}} \left [ |(\overline{v} - u ) \cdot \Omega [u] | + |w| \right ].
\end{align*}
As the function $w \mapsto \Exp{ - (1 - \lambda)  \frac{|w|^2}{2}} \left [ |(\overline{v} - u ) \cdot \Omega [u] | + |w| \right ]$ belongs to $L^1(\R^d)$, we deduce by dominated convergence that
\[
\lim _{\sigma \searrow 0 } \left \{ \partial _l Z (\sigma, |u|)  \sigma ^{1 - d/2} \exp \left ( \frac{\Phiu (\overline{v})}{\sigma} \right ) \right \} = (\overline{v} - u ) \cdot \Omega [u]\int _{\R^d} \Exp{- \frac{\partial _v ^2 \Phi _0 (\overline{v})w \cdot w}{2}}\;\mathrm{d}w.
\]
As we know, cf. Proposition \ref{Sign},  that $\inf_{|u|\in [\delta, r_0 - \delta] \cup [r_0 + \delta, +\infty[ } | (\overline{v} - u ) \cdot \Omega [u] | >0, 0 < \delta < r_0/2$, we deduce that for any $\delta \in ]0, r_0/2[$, there is $\sigma _\delta >0$ such that 
\[
\partial _l Z(\sigma, |u|)>0\;\mbox{ for any } 0 < \sigma < \sigma_\delta,\;\;\delta \leq |u| \leq r_0 - \delta
\]
and
\[
\partial _l Z(\sigma, |u|)<0\;\mbox{ for any } 0 < \sigma < \sigma_\delta,\;\;|u| \geq r_0 + \delta.
\]
Motivated by the above behavior of the function $Z$, we assume that the potential $v \mapsto V(|v|)$ satisfies \eqref{Equ11} (such that $Z$ is well defined) and belongs to the family $\mathcal{V}$ defined by: there exists $\sigma _0>0$ verifying
\begin{enumerate}
\item
For any $0 < \sigma < \sigma _0$ there is $l(\sigma) >0 $ such that $Z(\sigma,l)$ is stricly increasing on $[0,l(\sigma)]$ and strictly decreasing on $[l(\sigma), +\infty[$;
\item
For any $\sigma \geq \sigma _0, Z(\sigma,l)$ is strictly decreasing on $[0,+\infty[$. 
\end{enumerate}
In fact, the critical diffusion coefficient $\sigma_0$ vanishes the second order derivative of $Z$ with respect to $l$, at $l = 0$, as shown next.
\begin{pro}
\label{CriticalDiffusion}
Let $V(|\cdot|) \in \mathcal{V}$ be a potential satisfying \eqref{Equ11}. Then we have
\[
\partial _{ll} ^2 Z (\sigma, 0) \geq 0,\;\;0 < \sigma < \sigma _0,\;\;\partial _{ll} ^2 Z(\sigma _0, 0) = 0,\;\;\partial _{ll} ^2 Z (\sigma, 0) \leq 0,\;\;\sigma > \sigma _0
\]
and
\[
\partial _{ll} ^2 Z (\sigma, l(\sigma)) \leq 0,\;\;0 < \sigma < \sigma _0.
\]
\end{pro}
\begin{proof}
By Remark \ref{Modulus}  we know that $Z(\sigma, \cdot)$ possesses a second order derivative with respect to $l$. As $\partial _l Z(\sigma,0) = 0$, we write
\[
\frac{1}{2}\partial _{ll} ^2 Z(\sigma,0) = \lim _{l \searrow 0} \frac{Z(\sigma, l) - Z(\sigma,0) - l \partial _l Z(\sigma,0)}{l^2} = \lim _{l \searrow 0} \frac{Z(\sigma,l) - Z(\sigma,0)}{l^2}.
\]
We deduce that $\partial _{ll} ^2 Z(\sigma, 0) \geq 0$ for any $0 < \sigma \leq \sigma _0$ and $\partial _{ll} ^2 Z(\sigma, 0) \leq 0$ for any $\sigma \geq \sigma _0$. In particular $\partial _{ll} ^2 Z(\sigma_0,0) = 0$. For any $0 < \sigma < \sigma _0$, the function $Z(\sigma, \cdot)$ possesses a maximum at $l = l(\sigma)>0$ and therefore $\partial _{ll} ^2 Z(\sigma, l(\sigma)) \leq 0$. 
\end{proof}
It is also easily seen that $\lim _{\sigma \nearrow \sigma_0} l(\sigma) = 0$. Indeed, assume that there is $\eta >0$ and a sequence $(\sigma_n)_{n \geq 1} \nearrow \sigma _0$ such that $0 < \sigma _n < \sigma _0, l(\sigma_n) \geq \eta$ for any $n \geq 1$. We have
\[
Z(\sigma_n, l(\sigma_n)) \geq Z(\sigma _n, \eta) > Z(\sigma _n,0),\;\;n \geq 1.
\]
After passing to the limit when $n \to +\infty$, we obtain a contradiction
\[
Z(\sigma_0, \eta) \geq Z(\sigma _0, 0) > Z(\sigma _0, \eta)
\]
and therefore $\lim_{\sigma \nearrow \sigma_0} l(\sigma) = 0$. We  have proved that $\sigma \mapsto l(\sigma)$ is continuous. 

\begin{remark}
Given a potential $V(|\cdot|) \in \mathcal{V}$, then the unique bifurcation point from the disordered state happens at $\sigma_0$. In fact, if we define the function
$$
H(\sigma,l)=\intv{M_u (v) ( v \cdot \Omega - l ) }\,,
$$
as in \cite{BCCD16}. Then by \eqref{Equ12Bis}, we get $\sigma \partial _l Z(\sigma, l) = Z(\sigma, l) H(\sigma,l)$. By taking the derivative with respect to $l$, we obtain
$$
\partial_l H = \sigma \left(\frac{\partial^2 _{ll} Z}{Z} - \frac{(\partial _l Z)^2}{Z^2}\right) \,.
$$
Therefore, for the curve $l(\sigma)$ such that $H(\sigma,l(\sigma))=0$, we get $\partial_l H (\sigma_0,0)=0$. Using implicit differentiation and the continuity of the curves and the functions involved,
it is also easy to check that $\partial_\sigma H (\sigma_0,0)=0$. Therefore, to clarify the behavior of the two curves at $\sigma_0$, one needs to work more to compute the $\lim _{\sigma \nearrow \sigma_0}l^\prime (\sigma)$. In any case, this shows that $\sigma_0$ is the only bifurcation point from the manifold of disorder states $u=0$ for potentials $V(|\cdot|) \in \mathcal{V}$ without the need of applying the Crandall-Rabinowitz bifurcation theorem. It would be interesting to use Crandall-Rabinowitz for general potentials to identify more general conditions for bifurcations.
\end{remark}

In the last part of this section, we explore some properties of the potentials $V$ in the class $\mathcal{V}$. We show that under the hypothesis \eqref{Equ18}, we retrieve a weaker version of \eqref{Equ15}. 
\begin{pro}
Let $V(|\cdot|) \in \mathcal{V}$ be a potential satisfying \eqref{Equ11}. The application $\sigma \mapsto l(\sigma)$ is continuous on $\R_+ ^\star$. Moreover, if $V(|\cdot|) \in C^2(\R^d)$ verifies \eqref{Equ18} and there is the limit $\lims l(\sigma) = r_0 >0$, then
\[
V^\prime (r) \leq 0 \;\mbox{ for any } 0 < r \leq r_0 \;\mbox{ and } V^\prime (r) \geq 0 \;\mbox{ for any } r \geq r_0.
\]
\end{pro}
\begin{proof}
We are done if we check the continuity ant any $\sigma \in ]0,\sigma _0[$. Assume that there is a sequence $(\sigma _n)_{n \geq 1} \subset ]0,\sigma _0[$, $\lim_{n \to +\infty} \sigma _n = \sigma \in ]0,\sigma _0[$ and $\eta >0$ such that $l(\sigma _n ) > l(\sigma) + \eta$ for any $n \geq 1$. We have
\[
Z(\sigma_n, l(\sigma_n)) > Z(\sigma_n, l(\sigma) + \eta) > Z(\sigma _n, l(\sigma _n)),\;\;n \geq 1,
\]
leading to the contradiction
\[
Z(\sigma, l(\sigma) + \eta) \geq Z(\sigma, l(\sigma)) > Z(\sigma, l(\sigma) + \eta).
\]
Similarly, assume that there is a sequence $(\sigma _n)_{n \geq 1} \subset ]0,\sigma _0[$, $\lim_{n \to +\infty} \sigma _n = \sigma \in ]0,\sigma _0[$ and $\eta \in ]0,l(\sigma)[$ such that $l(\sigma _n ) < l(\sigma) - \eta$ for any $n \geq 1$. We have
\[
Z(\sigma _n, l (\sigma _n)) \geq Z(\sigma _n, l(\sigma) - \eta) > Z(\sigma _n, l(\sigma)),
\]
leading to the contradiction
\[
Z(\sigma, l(\sigma) - \eta) \geq Z(\sigma, l(\sigma)) > Z(\sigma, l(\sigma) - \eta).
\]
Therefore $\lim _{n \to +\infty} l(\sigma _n) = l(\sigma)$ for any sequence $(\sigma_n)_{n\geq 1}$, $\lim_{n \to +\infty} \sigma _n = \sigma \in ]0,\sigma_0[$. 

Assume now that $\lims l (\sigma) = r_0 >0$. For any $l \in ]0,r_0[$, we have $0 < l < l(\sigma)$ for $\sigma \in ]0,\sigma _0[$ small enough. As $Z(\sigma, \cdot)$ is strictly increasing on $[0,l(\sigma)]$, we deduce that $\partial _l Z(\sigma, l) >0$ for $\sigma$ small enough, and by \eqref{Equ17} it comes that
\[
\intw{\Exp{- \frac{\Phi _0 (\vb + \sqrt{\sigma} w ) - \Phi _0 (\vb) - \sqrt{\sigma} \nabla _v \Phi _0 (\vb) \cdot w}{\sigma}}( - V^\prime(|\vb|) + \sqrt{\sigma} w \cdot \Omega ) } >0,
\]
where $\vb$ is the minimum point of $\Phi _{l\Omega}$, that is $\vb = |\vb| \Omega, |\vb| + V^\prime (|\vb|) = l$. Passing to the limit when $\sigma \searrow 0$ yields
\[
\intw{\Exp{- \frac{\partial _v ^2 \Phi _0 (\vb) w \cdot w}{2} }} \;V^\prime (|\vb|) \leq 0,
\]
and therefore $V^\prime (|\vb|) \leq 0$. As before, \eqref{Equ18} implies \eqref{Equ14} and therefore $r \mapsto r + V^\prime (r)$ is strictly increasing on $\R_+$. We have $l - |\vb| = V^\prime (|\vb|) \leq 0$ and $l =  |\vb| + V^\prime (|\vb|) \geq l + V^\prime (l)$ saying that $V^\prime (l) \leq 0$ for any $l \in ]0,r_0[$, and also for $l = r_0$.

Consider now $l >r_0$. For $\sigma \in ]0,\sigma_0[$ small enough we have $l > l(\sigma)$ and therefore $\partial _l Z(\sigma, l) <0$. As before, \eqref{Equ17} leads to $l - |\vb| = V^\prime ( |\vb|) \geq 0$ and we have $l =  |\vb| + V^\prime (|\vb|) \leq l + V^\prime (l)$ saying that $V^\prime (l) \geq 0$ for any $l >r_0$, and also for $l = r_0$. In particular $r_0$ is a critical point of $V$.
\end{proof}
In the next result we analyze the behavior of $l(\sigma)$ for $\sigma$ small.
\begin{pro}
\label{SmallSigma}
Let $V (|\cdot|) \in \mathcal{V}$ be a potential satisfying \eqref{Equ11}, \eqref{Equ18}. If $V (|\cdot|) \in C_b ^3 (\R^d)$ and there is the limit $\lims l(\sigma) = r_0 >0$, then we have for any $\Omega \in \sphere$
\[
V^{\prime \prime} (r_0) \lims \frac{l(\sigma) - r_0}{\sigma} = - \frac{1 + V^{\prime \prime} (r_0)}{6} \frac{\intw{( w \cdot \Omega) \partial _v ^3 \Phi _0 (r_0 \Omega) (w,w,w)\Exp{- \frac{\partial _v ^2 \Phi _0 (r_0 \Omega) w \cdot w}{2} }}}{\intw{\Exp{- \frac{\partial _v ^2 \Phi _0 (r_0 \Omega) w \cdot w}{2} }}}
\]
where $ \partial _v ^3 \Phi _0 (r_0 \Omega) (w,w,w) = \sum _{1\leq i,j,k\leq d} \frac{\partial ^3 \Phi _0 }{\partial _{v_k} \partial _{v_j} \partial _{v_i}} (r_0 \Omega) w_k w_j w_i$.
\end{pro}
\begin{proof}
We fix $\Omega \in \sphere$. For any $\sigma \in ]0,\sigma _0[$ we have $\partial _l Z(\sigma, l(\sigma)) = 0$, and \eqref{Equ17} implies
\begin{equation}
\label{Equ71}
\intw{\Exp{- \frac{\Phi _0 (\vb + \sqrt{\sigma} w ) - \Phi _0 (\vb) - \sqrt{\sigma} \nabla _v \Phi _0 (\vb) \cdot w}{\sigma}}( - V^\prime(|\vb|) + \sqrt{\sigma} w \cdot \Omega )} = 0,
\end{equation}
where  $\vb$ is the minimum point of $\Phi _{l(\sigma)\Omega}$, that is $\vb = |\vb| \Omega, |\vb| + V^\prime (|\vb|) = l(\sigma)$. As the function $r \mapsto r + V^\prime (r)$ is strictly increasing on $\R_+$, when $\sigma \searrow 0$, we have $l(\sigma) \to r_0$ and $|\vb|$ converges toward the reciprocal image of $r_0$, through the function $r \mapsto r + V^\prime(r)$, which is $r_0$. We deduce
\[
\frac{l(\sigma) - r_0}{\sigma} = \frac{|\vb| - r_0 }{\sigma} + \frac{V^\prime(|\vb|) - V^\prime(r_0)}{|\vb| - r_0} \;\frac{|\vb| - r_0}{\sigma} ,
\]
implying that 
\[
\lims \frac{l(\sigma) - r_0}{\sigma} =( 1 + V ^{\prime \prime} (r_0 )) \lims \frac{|\vb| - r_0}{\sigma}.
\]
We will compute 
\[
\lims \frac{V^\prime (|\vb|)}{\sigma} = V ^{\prime \prime} (r_0 ) \lims \frac{|\vb| - r_0 }{\sigma}.
\]
Thanks to \eqref{Equ71} we have
\begin{align}
\label{Equ72}
& \intw{\Exp{- \frac{\partial _v ^2 \Phi _0 (r_0 \Omega) w\cdot w}{2}}}\;\lims \frac{V^\prime (|\vb|)}{\sigma}\\
&  = \lims \intw{\Exp{- \frac{\Phi _0 (\vb + \sqrt{\sigma} w ) - \Phi _0 (\vb) - \sqrt{\sigma} \nabla _v \Phi _0 (\vb) \cdot w}{\sigma}}\frac{ w \cdot \Omega }{\sqrt{\sigma}} }.\nonumber 
\end{align}
Observe that 
\begin{align}
\label{Equ73}
\intw{& \Exp{- \frac{\Phi _0 (\vb + \sqrt{\sigma} w ) - \Phi _0 (\vb) - \sqrt{\sigma} \nabla _v \Phi _0 (\vb) \cdot w}{\sigma}}\frac{ w \cdot \Omega }{\sqrt{\sigma}} } \\
& = \intw{\left [ \Exp{- \frac{\Phi _0 (\vb + \sqrt{\sigma} w ) - \Phi _0 (\vb) - \sqrt{\sigma} \nabla _v \Phi _0 (\vb) \cdot w}{\sigma}} \right. \nonumber \\
& \left. \quad -  \Exp{- \frac{\partial _v ^2 \Phi _0 (\vb) w\cdot w}{2}} \right ] \frac{ w \cdot \Omega }{\sqrt{\sigma}}} \nonumber 
\end{align}
and 
\begin{align*}
\lims & \frac{1}{\sqrt{\sigma}}\left [ \Exp{- \frac{\Phi _0 (\vb + \sqrt{\sigma} w ) - \Phi _0 (\vb) - \sqrt{\sigma} \nabla _v \Phi _0 (\vb) \cdot w}{\sigma}}  -  \Exp{- \frac{\partial _v ^2 \Phi _0 (\vb) w\cdot w}{2}} \right ]\\
& = - \Exp{- \frac{\partial _v ^2 \Phi _0 (r_0 \Omega) w\cdot w}{2}} \\
& \quad \times \lims \frac{\Phi _0 (\vb + \sqrt{\sigma} w ) - \Phi _0 (\vb) - \sqrt{\sigma} \nabla _v \Phi _0 (\vb) \cdot w - \frac{\sigma}{2} \;\partial _v ^2 \Phi _0 (\vb)w\cdot w}{\sigma ^{3/2}}\\
& = - \Exp{- \frac{\partial _v ^2 \Phi _0 (r_0 \Omega) w\cdot w}{2}}\lims \frac{1}{\sqrt{\sigma}} \int _0 ^1 (1-t) [ \partial _v ^2 \Phi _0 (\vb + t \sqrt{\sigma} w) -  \partial _v ^2 \Phi _0 (\vb)]w\cdot w\;\md t\\
& = - \frac{1}{6} \partial _v ^3 \Phi _0 (r_0 \Omega) (w,w,w) \Exp{- \frac{\partial _v ^2 \Phi _0 (r_0 \Omega) w\cdot w}{2}}.
\end{align*}
Recall that, thanks to \eqref{Equ18}, we have $\partial _v ^2 \Phi _0 (v) \geq (1- \lambda ) I_d, v \in \R^d$, implying that 
\[
\frac{\Phi _0 (\vb + \sqrt{\sigma} w ) - \Phi _0 (\vb) - \sqrt{\sigma} \nabla _v \Phi _0 (\vb) \cdot w}{\sigma} = \int _0 ^1 (1-t) \partial _v ^2 \Phi _0 (\vb + t \sqrt{\sigma} w) w \cdot w \;\md t \geq (1 - \lambda) \frac{|w|^2}{2}
\]
and
\[
\frac{\partial _v ^2 \Phi _0 (\vb) w \cdot w}{2} \geq (1 - \lambda) \frac{|w|^2}{2}, \;\;w \in \R^d.
\]
Therefore the integrand of the right hand side in \eqref{Equ73} can be bounded, uniformly with respect to $\sigma>0$ by a $L^1$ function
\begin{align*}
& \left | \Exp{- \frac{\Phi _0 (\vb + \sqrt{\sigma} w ) - \Phi _0 (\vb) - \sqrt{\sigma} \nabla _v \Phi _0 (\vb) \cdot w}{\sigma}}  - \Exp{-\frac{\partial _v ^2 \Phi _0 (\vb) w \cdot w}{2}}\right |\frac{|(w \cdot \Omega)|}{\sqrt{\sigma}}\\
& \leq  \Exp{- (1 - \lambda) \frac{|w|^2}{2}} \frac{|(w \cdot \Omega)|}{\sqrt{\sigma}}\int _0 ^1 (1-t) [  \partial _v ^2 \Phi _0 (\vb + t \sqrt{\sigma} w) - \partial _v ^2 \Phi _0 (\vb) ]w \cdot w\;\md t\\
& \leq \|V(|\cdot |)\|_{C_b ^3 (\R^d)} |w|^2 \Exp{- (1 - \lambda) \frac{|w|^2}{2}},\;\;w \in \R^d.
\end{align*}
Combining \eqref{Equ72}, \eqref{Equ73}, we obtain by dominated convergence 
\begin{align*}
V^{\prime \prime}(r_0) \lims \frac{|\vb| - r_0 }{\sigma} & = \lims \frac{V^\prime(|\vb|)}{\sigma} \\
& = - \frac{\intw{(w \cdot \Omega) \partial _v ^3 \Phi _0 (r_0 \Omega) (w,w,w) 
\Exp{ - \frac{\partial _v ^2 \Phi _0 (r_0 \Omega) w \cdot w}{2}}}   }{6 \intw{\Exp{ - \frac{\partial _v ^2 \Phi _0 (r_0 \Omega) w \cdot w}{2}}}   }
\end{align*}
and therefore
\begin{align*}
V^{\prime \prime}(r_0)&  \lims \frac{l(\sigma) - r_0}{\sigma}  = ( 1 + V^{\prime \prime}(r_0) )\;V^{\prime \prime}(r_0) \lims \frac{|\vb| - r_0 }{\sigma} \\
& = - \frac{1 + V^{\prime \prime}(r_0)}{6}\frac{\intw{(w \cdot \Omega) \partial _v ^3 \Phi _0 (r_0 \Omega) (w,w,w) 
\Exp{ - \frac{\partial _v ^2 \Phi _0 (r_0 \Omega) w \cdot w}{2}}}   }{\intw{\Exp{ - \frac{\partial _v ^2 \Phi _0 (r_0 \Omega) w \cdot w}{2}}}}.
\end{align*}
\end{proof}

\section{Linearization of the interaction mechanism}
\label{LinIntMec}
We intend to investigate the asymptotic behavior of \eqref{Equ3} when $\eps \searrow 0$. We introduce the formal development
\[
\fe = f + \eps \fo + ...
\]
and we expect that
$
Q(f) = 0
$ and
\begin{equation}
\label{Equ21} \partial _t f + v \cdot \nabla _x f = \lime \frac{Q(\fe) - Q(f)}{\eps} = \mathrm{d}Q_f (\fo) = : \calL _f (\fo).
\end{equation}
As seen before, for any $(t,x) \in \R_+ \times \R^d$, the individual density $f(t,x,\cdot)$ is a von Mises-Fisher distribution
\[
f(t,x,v) = \rho (t,x) M_{|u|\Omega (t,x)} (v),\;\;v \in \R^d
\]
where $|u|$ is a critical point of $Z(\sigma, \cdot)$, that is
\[
|u| \in \{0, l(\sigma)\}\;\mbox{ if } 0 < \sigma < \sigma _0\;\mbox{ and }\;|u| = 0\;\mbox{ if } \; \sigma \geq \sigma _0. 
\]
It remains to determine the fluid equations satisfied by the macroscopic quantities $\rho, \Omega$. When $|u| = 0$, the continuity equation leads to $\partial _t \rho = 0$. In the sequel we concentrate on the case $|u| = l(\sigma), 0 < \sigma < \sigma_0$ (that is, the modulus of the mean velocity is given, as a function of $\sigma$). We follow the strategy in \cite{BosCar15, AceBosCarDeg19}. We consider
\[
\ltmu = \{ \chi : \R^d \to \R \mbox{ measurable },\;\intv{(\chi (v))^2 M_u(v)} < + \infty \}
\]
and
\[
\homu = \{ \chi : \R^d \to \R \mbox{ measurable },\;\intv{[\;(\chi (v))^2 + |\nabla _v \chi |^2\;]M_u(v)} < + \infty \}.
\]
We introduce the usual scalar products 
\[
(\chi, \theta)_{M_u} = \intv{\chi (v) \theta (v) M_u (v)},\;\;\chi, \theta \in \ltmu ,
\]
\[
((\chi, \theta))_{M_u} = \intv{( \chi (v) \theta (v) + \nabla _v \chi \cdot \nabla _v \theta) M_u (v)},\;\;\chi, \theta \in \homu
\]
and we denote by $|\cdot |_{M_u}, \|\cdot \|_{M_u}$ the associated norms. Moreover we need a Poincar\'e inequality. This comes from the equivalence between the Fokker-Planck and Schr\"odinger operators. As described in \cite{BonCarGouPav16}, we can write it as
\[
- \frac{\sigma}{\sqrt{M_u}} \Divv\left ( M_u \nabla _v \left ( \frac{g}{\sqrt{M_u}}\right ) \right ) = - \sigma \Delta _v g + \left [ \frac{1}{4\sigma} |\nabla _v \Phiu |^2 - \frac{1}{2} \Delta _v \Phiu \right ] g.
\]
The operator $\mathcal{H}_u = - \sigma \Delta _v + \left [ \frac{1}{4\sigma} |\nabla _v \Phiu |^2 - \frac{1}{2} \Delta _v \Phiu \right ]$ is defined in the domain
\[
D(\mathcal{H}_u) = \left \{g \in L^2 (\R^d),\;\left [ \frac{1}{4\sigma} |\nabla _v \Phiu |^2 - \frac{1}{2} \Delta _v \Phiu \right ]g \in L^2 (\R^d),\;\;\Delta _v g \in L^2 (\R^d)  \right \}.
\]
We have a spectral decomposition of the operator $\mathcal{H}_u$ under suitable confining assumptions (cf. Theorem XIII.67 in \cite{ReedSimon78}).
\begin{lemma}
\label{RSVol4}
Assume that the function $v \mapsto \frac{1}{4\sigma} |\nabla _v \Phiu |^2 - \frac{1}{2}\Delta _v \Phiu$ belongs to $L^1_{\mathrm{loc}}(\R^d)$, is bounded from below and is coercive {\it i.e.}
\[
\lim _{|v| \to + \infty} \left [ \frac{1}{4\sigma} |\nabla _v \Phiu |^2 - \frac{1}{2}\Delta _v \Phiu \right ]=  + \infty.
\]
Then $\mathcal{H}_u ^{-1}$ is a self adjoint compact operator in $L^2(\R^d)$ and $\mathcal{H}_u$ admits a spectral 
decomposition, that is, a nondecreasing sequence of real numbers $(\lambda _u ^n)_{n\in \N}$, $\lim _{n \to + \infty} \lambda _u ^n = + \infty$, and a $L^2(\R^d)$-orthonormal basis $(\psi _u ^n)_{n \in \N}$ such that $\mathcal{H}_u \psi _u ^n = \lambda _u ^n \psi _u ^n, n \in \N$, $\lambda _u ^0 = 0$, $\lambda _u ^1 >0$. 
\end{lemma}
Therefore, under the hypotheses in Lemma \ref{RSVol4}, for any $u \in \R^d$ there is $\lambda _u >0$ such that for any $\chi \in \homu$ we have
\begin{equation}
\label{Equ22}
\sigma \intv{|\nabla _v \chi |^2 M_u (v)} \geq \lambda _u \intv{\left |\chi (v) - \intvp{\chi (\vp) M_u (\vp)}\right |^2 M_u (v)}.
\end{equation}
The fluid equations are obtained by taking the scalar product of \eqref{Equ21} with elements in the kernel of the (formal) adjoint of $\calL _f$, that is with functions $\psi = \psi (v)$ such that 
\[
\intv{(\calL _f g)(v) \psi (v)} = 0,\;\;\mbox{ for any function } \;g = g(v),
\]
see also \cite{BarGolLevI93,BarGolLevII93,BosDCDS15,BosIHP15,Lev93,Lev96}. For example, $\psi = 1$ belongs to the kernel of $\calL_f ^\star$
\[
\intv{(\calL_f g )(v)} = \intv{\lime \frac{Q(f + \eps g) - Q(f)}{\eps}} = \lime \frac{1}{\eps} \intv{\{Q(f + \eps g) - Q(f)\}} = 0,
\]
and we obtain the continuity equation \eqref{Equ51}
\begin{equation*}
\partial _t \intv{f} + \Divx \intv{f v } = \intv{\calL_f (f^1)} = 0.
\end{equation*}
In the sequel we determine the formal adjoint of the linearization of the collision operator $Q$ around its equilibria.
\begin{pro}
\label{Linear}
Let $f = f(v)$ be an equilibrium with non vanishing mean velocity
\[
f = \rho M_u,\;\;\rho = \rho[f],\;\;u = |u|\Omega [f],\;\;|u| = l(\sigma),\;\;0 < \sigma < \sigma _0.
\]
\begin{enumerate}
\item
The linearization $\calL _f = \mathrm{d} Q_f$ is given by
\[
\calL _f g = \Divv \left \{ \sigma \nabla _v g + g \nabla _v \Phiu - M_u \intvp{(\vp - u) g(\vp)} \right \}.
\] 
\item
The formal adjoint of $\calL _f$ is
\begin{equation*}
\calL _f ^\star \psi = \sigma \frac{\Divv (M_u \nabla _v \psi)}{M_u} + (v - u)  \cdot W[\psi],\;\;W[\psi] := \intv{M_u (v) \nabla _v \psi }.
\end{equation*}
\item
We have the identity
\[
\calL _f (f(v-u)) = \sigma \nabla _v f - \Divv \left ( f \calmu \right ),\;\;\calmu := \intvp{M_u (\vp) (\vp - u) \otimes ( \vp - u)}.
\]
\end{enumerate}
\end{pro}
\begin{proof}$\;$\\
1. We have
\[
\calL _f g  = \frac{\md}{\md s} \Big |_{s= 0} Q(f+sg) = \Divv \left \{\sigma \nabla _v g + g \nabla _v \Phiu  - f \frac{\md}{\md s}\Big|_{s= 0} u[f+sg]  \right \}
\]
and
\[
\frac{\md}{\md s}\Big|_{s= 0}u [f+sg] = \frac{\intv{\;(v - u[f])g(v)}}{\intv{f(v)} } .
\]
Therefore we obtain
\[
\calL _f g = \Divv \left \{ \sigma \nabla _v g + g \nabla _v \Phiu - - M_u \intvp{\;(\vp - u[f])g(\vp)} \right \}. 
\]
2. We have
\begin{align*}
& \intv{(\calL _f g) (v)\psi(v)  }  = -  \intv{\left \{ \sigma \nabla _v g + g \nabla _v \Phiu - M_u (v) \intvp{\;(\vp - u[f]) g(\vp) }\right \}\cdot \nabla _v \psi }\\
& = \intv{g (v) \left (  \sigma \Divv \nabla _v \psi - \nabla _v \psi \cdot \nabla _v \Phiu \right ) } + \intvp{g(\vp) ( \vp - u[f]) }\cdot \intv{M_u (v) \nabla _v \psi} \\
& = \intv{g(v) \left (\sigma \Divv \nabla _v \psi - \nabla _v \psi \cdot \nabla _v \Phiu + (v - u[f]) \cdot W[\psi]  \right ) }
\end{align*}
implying 
\[
\calL _f ^\star \psi = \sigma \frac{\Divv (M_u \nabla _v \psi)}{M_u} + (v - u[f])  \cdot W[\psi].
\]
3. For any $i \in \{1,...,d\}$ we have
\begin{align*}
\calL _f (f(v-u)_i)& = \Divv \left [(v - u)_i (\underbrace{\sigma \nabla _v f + f \nabla _v \Phiu}_{= 0}) + \sigma f e_i - M_u \intvp{(\vp - u)_i (\vp - u) f(\vp)} \right ] \\
& = \sigma \partial _{v_i} f - \Divv\left (f \intvp{(\vp - u ) \otimes (\vp - u) M_u (\vp) }   \right ) _i
\end{align*}
and therefore
\begin{align*}
\calL _f (f(v-u)) &  = \sigma \nabla _v f - \Divv ( f \calmu ).
\end{align*}
\end{proof}
We identify now the kernel of $\calL _f ^\star$.
\begin{lemma}
\label{AdjointRes}
Let $f = \rho M_u >0$ be an equilibrium with non vanishing mean velocity. The following statements are equivalent
\begin{enumerate}
\item
The function $\psi = \psi (v)$ belongs to $\ker \calL _f ^\star$.

\item
The function $\psi = \psi (v)$ satisfies
\begin{equation}
\label{Equ24}
\sigma \frac{\Divv (M_u \nabla _v \psi)}{M_u (v)} + (v - u ) \cdot W = 0
\end{equation}
for some vector $W \in \ker ( \calmu - \sigma I_d)$.
\end{enumerate}
Moreover, the linear map $W : \ker \calL _f ^\star \to \ker ( \calmu - \sigma I_d)$, defined by $W[\psi] = \intv{M_u (v) \nabla _v \psi }$ induces an isomorphism between the vector spaces $\ker \calL _f ^\star /\ker W$ and $\ker ( \calmu - \sigma I_d)$, where $\ker W$ is the set of constant functions. 
\end{lemma}
\begin{proof} $\;$\\
1.$\implies$2. Let $\psi$ be an element of $\ker \calL _f ^\star$.  By the last statement in Proposition \ref{Linear} we deduce 
\begin{align*}
0 & = \intv{\calL _f ^\star \psi  \; f (v - u)} = \intv{\psi (v) \calL _f (f(v-u))} \\
& = \intv{\psi (v) \left [ \sigma \nabla _v f - \Divv  ( f \calmu )  \right ] }\\
& = - \sigma \intv{f(v) \nabla _v \psi } + \calmu \intv{f(v) \nabla _v \psi }\\
& = \rho ( \calmu - \sigma I_d ) W[\psi].
\end{align*}
As $\rho >0$ we deduce that $W[\psi] \in \ker ( \calmu - \sigma I_d)$ and by the second statement in Proposition \ref{Linear} it comes that
\[
\sigma \frac{\Divv (M_u \nabla _v \psi)}{M_u (v)} + (v - u ) \cdot W = 0,\;\;W = W[\psi] \in \ker ( \calmu - \sigma I_d).
\]
\noindent
2.$\implies$1. Let $\psi$ be a function satisfying \eqref{Equ24} for some vector $W \in \ker (\calmu - \sigma I_d)$. Multiplying by $M_u(v) (v-u)$ and integrating with respect to $v$ yields 
\[
- \sigma \intv{M_u (v) \nabla _v \psi } +\calmu W = 0.
\]
As we know that $W \in \ker (\calmu - \sigma I_d)$, we deduce that $W = W[\psi]$, implying that $\psi$ belongs to $\ker \calL _f ^\star$
\[
\calL _f ^\star \psi = \sigma \frac{\Divv (M_u \nabla _v \psi)}{M_u} + (v- u) \cdot W[\psi] = \sigma \frac{\Divv (M_u \nabla _v \psi)}{M_u} +( v - u) \cdot W = 0.
\]
\end{proof}
We focus on the eigenspace $\ker ( \calmu - \sigma I_d)$.
\begin{lemma} 
\label{Spectral}
Let $M_u$ be an equilibrium with non vanishing mean velocity. Then we have
\[
\calmu - \sigma I_d = \sigma ^2 \frac{\partial _{ll} ^2 Z(\sigma, l(\sigma))}{Z(\sigma, l(\sigma))} \Omega \otimes \Omega \leq 0,\;\;\Omega = \frac{u}{|u|}. 
\]
In particular $(\R u ) ^\perp \subset \ker (\calmu - \sigma  I_d)$ with equality iff $\partial _{ll} ^2 Z(\sigma, l(\sigma)) \neq 0$. 
\end{lemma}
\begin{proof}
Let us consider $\{E_1, \ldots, E_{d-1}\}$ an orthonormal basis of $(\R\Omega) ^\perp$. By using the decomposition 
\[
v - u = ( \Omega \otimes \Omega ) (v-u) + \sumi ( E_i \otimes E_i ) (v-u) = ( \Omega \otimes \Omega ) (v-u) + \sumi ( E_i \otimes E_i ) v
\]
we obtain
\begin{align*}
\calmu & = \intv{\left [  \Omega \otimes \Omega  (v-u) + \sumi  E_i \otimes E_i  v
\right ]\otimes \left [ \Omega \otimes \Omega (v-u) + \sumj  E_j \otimes E_j  v \right ] M_u (v) } \\
& = ( \calmu \Omega \cdot \Omega ) \Omega \otimes \Omega + \sumi (\calmu E_i \cdot E_i ) E_i \otimes E_i
\end{align*}
since we have
\begin{equation}
\label{Equ26}
\calmu \Omega \cdot E_j = 0,\;\;1 \leq j \leq d-1
\end{equation}
and
\begin{equation}
\label{Equ27}
\calmu E_i \cdot E_j = \delta _{ij} \intv{\frac{|v|^2 - (v \cdot \Omega) ^2 }{d-1} M_u (v)},\;\;1\leq i, j \leq d-1. 
\end{equation}
The formula \eqref{Equ26} comes by the change of variable $v = (I_d - 2 E_j \otimes E_j)\vp$, by noticing that $I_d - 2 E_j \otimes E_j \in \calTu$, for any $1 \leq j \leq d-1$
\begin{align*}
\calmu \Omega \cdot E_j & = \intv{\Omega \cdot (v - u) (E_j \cdot v) M_u (v)}\\
& = - \intvp{\Omega \cdot (\vp - u) (E_j \cdot \vp) M_u (\vp) }\\
&= - \calmu \Omega \cdot E_j = 0,\;\;1 \leq j \leq d-1.
\end{align*}
For the formula \eqref{Equ27} with $i \neq j$ we use the rotation $\calO _{ij} \in \calTu$
\[
v = \calO _{ij}\vp,\;\;\calO _{ij} = \Omega \otimes \Omega + \sum_{k \notin\{i,j\}}E_k \otimes E_k + E_i \otimes E_j - E_j \otimes E_i.
\]
Notice that 
\[
(E_i \cdot v) (E_j \cdot v) = - (E_j \cdot \vp) (E_i \cdot \vp),\;\;(E_i \cdot v) ^2 = (E_j \cdot \vp)^2
\]
and therefore,
\begin{align*}
\calmu E_i \cdot E_j & = \intv{(E_i \cdot v) (E_j \cdot v) M_u (v)} \\
& = - \intvp{(E_j \cdot \vp) (E_i \cdot \vp) M_u (\vp)} \\
& = - \calmu E_i \cdot E_j = 0,\;\;1\leq i\neq j \leq d-1
\end{align*}
and
\begin{align*}
\calmu E_i \cdot E_i & = \intv{(E_i \cdot v)^2  M_u (v)} \\
& =  \intvp{(E_j \cdot \vp)^2  M_u (\vp)} \\
& = \calmu E_j \cdot E_j = 0,\;\;1\leq i, j \leq d-1 .
\end{align*}
As $\sumi (E_i \cdot v)^2 = |v|^2 - (v \cdot \Omega) ^2$, we obtain
\[
\intv{(E_i \cdot v) ^2 M_u (v)} = \intv{\frac{|v|^2 - (v \cdot \Omega)^2}{d-1} M_u (v) },\;\;1 \leq i \leq d-1
\]
and
\[
\calmu = \intv{( \; (v-u) \cdot \Omega )^2 M_u (v) } \;\Omega \otimes \Omega + \intv{\frac{|v|^2 - (v \cdot \Omega)^2}{d-1} M_u (v) }\; (I_d - \Omega \otimes \Omega).
\]
We claim that $\intv{\frac{|v|^2 - (v \cdot \Omega)^2}{d-1} M_u (v) } = \sigma$. Multiplying $\sigma \nabla _v M_u + M_u (v) \nabla _v \Phiu = 0$ by $(|v|^2 I_d - v \otimes v ) \Omega$ we obtain
\[
\intv{\sigma \nabla _v M_u \cdot ( |v|^2 I_d - v \otimes v ) \Omega } + \intv{M_u (v) \nabla _v \Phiu \cdot ( |v|^2 I_d - v \otimes v ) \Omega } = 0.
\]
But we have
\[
\Divv [ ( |v|^2 I_d - v \otimes v ) \Omega]  = \Divv [|v|^2 \Omega - (v \cdot \Omega) v ] = -(d-1) (v \cdot \Omega) 
\]
and 
\[
\nabla _v \Phiu \cdot  ( |v|^2 I_d - v \otimes v ) \Omega = \left ( v - u + V^\prime (|v|) \frac{v}{|v|} \right ) \cdot  ( |v|^2 I_d - v \otimes v ) \Omega = - (|v|^2 - (v \cdot \Omega) ^2 ) |u|.
\]
We deduce that 
\[
(d-1) \sigma \underbrace{\intv{(v \cdot \Omega) M_u (v)}}_{=|u|} - |u| \intv{[|v|^2 - (v \cdot \Omega) ^2 ] M_u (v)} = 0
\]
and by taking into account that $|u| = \intv{(v \cdot \Omega) M_u (v)}$, we obtain
\[
\intv{\frac{|v|^2 - (v \cdot \Omega)^2}{d-1} M_u (v) } = \sigma.
\]
By Remark \ref{Modulus}, we know that 
\[
\sigma ^2 \frac{\partial _{ll} ^2 Z(\sigma, l(\sigma))}{Z(\sigma, l(\sigma))} = \intv{M_u (v) \{ ((v-u) \cdot \Omega ) ^2 - \sigma \}} = \intv{M_u (v) ((v-u) \cdot \Omega ) ^2} - \sigma
\]
and finally we have
\[
\calmu - \sigma I_d = \left ( \intv{((v-u) \cdot \Omega ) ^2 M_u (v)} - \sigma \right ) \Omega \otimes \Omega = \sigma ^2 \frac{\partial _{ll} ^2 Z(\sigma, l(\sigma))}{Z(\sigma, l(\sigma))} \;\Omega \otimes \Omega.
\]
As $l(\sigma)$ is a maximum point of $Z(\sigma, \cdot)$, we have $\partial _{ll} ^2 Z(\sigma, l(\sigma)) \leq 0$ and therefore $\calmu \leq \sigma I_d$. 
\end{proof}

\section{The kernel of $\calL _f ^\star$}
\label{Invariants}

By Lemmas \ref{AdjointRes}, \ref{Spectral}, any solution of \eqref{Equ24} with $W \in (\R u ) ^\perp$ belongs to the kernel of the formal adjoint $\calL_f ^\star$. Generally we will solve the elliptic problem
\begin{equation}
\label{Equ31} 
- \sigma \Divv (M_u \nabla _v \psi ) = (v - u) \cdot W M_u (v),\;\;v \in \R^d
\end{equation}
for any $W \in \R^d$. We consider the continuous bilinear symmetric form $a_u :\homu \times \homu \to \R$ defined by
\[
a_u (\varphi, \theta) = \sigma \intv{\nabla _v \varphi \cdot \nabla _v \theta M_u (v)},\;\;\varphi, \theta \in \homu
\]
and the linear form $L : \homu \to \R, L(\theta) = \intv{\theta (v) (v-u ) \cdot W M_u (v) },\; \theta \in \homu$. Notice that under the hypothesis \eqref{Equ11} $L$ is bounded on $\homu$
\begin{align*}
\intv{|\theta (v) (v - u ) \cdot W |M_u } & \leq \left ( \intv{(\theta (v))^2 M_u } \right ) ^{1/2} \left ( \intv{(|v| + |u|) ^2 M_u }\right ) ^{1/2} |W|.
\end{align*}
We are looking for variational solutions of \eqref{Equ31} {\it i.e.,}
\begin{equation}
\label{Equ30}
\psi \in \homu \;\mbox{ and } a_u (\psi, \theta) = L(\theta)\;\;\mbox{ for any }\theta \in \homu.
\end{equation}
When taking $\theta = 1 \in \homu$, we obtain the following necessary condition for the solvability of \eqref{Equ31}
\begin{equation}
\label{Equ32}
L(1) = \intv{(v-u) \cdot W M_u (v) } = 0
\end{equation}
which is satisfied for any $W \in \R^d$, because $M_u$ has mean velocity $u$. It happens that \eqref{Equ32} also guarantees the solvability of \eqref{Equ31}. For that, it is enough to observe that the bilinear form $a_u$ is coercive on the Hilbert space $\thomu : = \{ \theta \in \homu\;:\; ((\theta, 1))_{M_u} = 0\}$. Indeed, for any $\theta \in \homu$ such that $((\theta, 1))_{M_u} = 0$, we have thanks to the Poincar\'e inequality \eqref{Equ22}
\[
\sigma \intv{|\nabla _v \theta|^2 M_u (v) } \geq \lambda _u \intv{(\theta (v))^2 M_u (v)},
\]
and therefore
\[
a_u (\theta, \theta) \geq  \frac{\lambda _u}{2} \intv{(\theta (v))^2 M_u (v)} + \frac{\sigma }{2} \intv{|\nabla _v \theta|^2 M_u (v)} \geq \frac{\min \{\sigma, \lambda _u\}}{2} \|\theta \|^2 _{M_u}.
\]
Thanks to Lax-Milgram lemma on the Hilbert space $\thomu$, there is a unique function $\psi \in \thomu$ such that
\begin{equation}
\label{Equ33} 
a_u (\psi, \tilde{\theta} ) = L(\tilde{\theta})\;\;\mbox{ for any } \tilde{\theta} \in \thomu. 
\end{equation}
The condition \eqref{Equ32} allows us to extend \eqref{Equ33} to $\homu$ (apply \eqref{Equ33} with $\tilde{\theta} = \theta - ((\theta, 1))_{M_u}$, for any $\theta \in \homu$). The uniqueness of the solution of \eqref{Equ33} implies the uniqueness, up to a constant, for the solution of \eqref{Equ30}. 

From now on, for any $W \in \R^d$, we denote by $\psi _W$ the unique solution of \eqref{Equ30}, verifying $\intv{\psi _W (v) M_u (v)} = 0$. Notice that $\psi _0 = 0$. The solution $\psi _W$ depends linearly on $W \in \R^d$. Let us introduce the Hilbert spaces
\[
\bltmu= \{\xi : \R^d \to \R^d\;\mbox{ measurable }, \sum _{i = 1} ^d \intv{(\xi _i (v))^2 M_u (v)} < +\infty\}
\]
\[
\bhomu= \{\xi : \R^d \to \R^d\;\mbox{ measurable }, \sum _{i = 1} ^d \intv{\{ (\xi _i (v))^2 + |\nabla _v \xi _i |^2\}M_u (v)} < +\infty\}
\]
endowed with the scalar product 
\[
(\xi, \eta)_{M_u} = \sum _{i = 1} ^d \intv{\xi _i (v) \eta _i (v) M_u (v)},\;\;\xi, \eta \in \bltmu
\]
\[
((\xi, \eta))_{M_u} = \sum _{i = 1} ^d \intv{\{ \xi _i (v) \eta _i (v) + \nabla _v \xi _i \cdot \nabla _v \eta _i \}M_u (v)},\;\;\xi, \eta \in \bhomu.
\]
We denote the induced norms by $|\xi |_{M_u} = (\xi, \xi ) _{M_u} ^{1/2}, \;\xi \in \bltmu$ and $\|\xi \|_{M_u} = ((\xi, \xi )) _{M_u} ^{1/2}, \;\xi \in \bhomu$. Obviously, a vector field $\xi = \xi (v)$ belongs to $\bhomu$ iff $\xi _i \in \homu$ for any $i \in \{1,...,d\}$ and we have
\[
\|\xi \|_{M_u}^2 = \sum _{i = 1} ^d \|\xi _i \|_{M_u} ^2. 
\]
Let us consider the closed subspace
\[
\bthomu = \{ \xi \in \bhomu\;:\;\intv{\xi (v) M_u (v)} = 0\}.
\]
Thanks to \eqref{Equ22}, for any $\xi \in \bhomu$ we have the inequality
\[
\sigma \sum _{i = 1} ^d \intv{|\nabla _v \xi _i |^2 M_u (v) } \geq \lambda _u \sum _{i = 1} ^d \intv{\left |\xi _i (v) - \intvp{\xi _i (\vp) M_u (\vp)}\right | ^ 2 M_u (v)}
\]
and therefore
\begin{align}
\label{Equ41} \sigma \sum _{i = 1} ^d \intv{|\nabla _v \xi _i |^2 M_u (v) } & \geq \frac{\min\{\sigma, \lambda _u \}}{2} \sum _{i = 1} ^d \intv{[(\xi _i (v))^2 + |\nabla _v \xi _i |^2 ] M_u (v)} \nonumber \\
& = \frac{\min\{\sigma, \lambda _u \}}{2}\|\xi \|^2 _{M_u},\;\;\xi \in \bthomu.
\end{align}
We introduce the continuous bilinear symmetric form ${\bf a} _u : \bhomu \times \bhomu \to \R$ defined by 
\[
{\bf a}_u (\xi, \eta) = \sigma \intv{\partial _v \xi : \partial _v \eta \; M_u (v)} = \sum _{i = 1} ^d a_u (\xi _i, \eta _i),\;\;\xi, \eta \in \bhomu
\]
and the linear form ${\bf L } : \bhomu \to \R$, ${\bf L }(\eta) = \intv{(v-u) \cdot \eta (v) M_u (v)}, \eta \in \bhomu$. Under the hypothesis \eqref{Equ11}, it is easily seen that ${\bf L}$ is bounded on $\bhomu$
\[
\intv{|(v - u) \cdot \eta (v)|M_u (v)} \leq \left ( \intv{(|v| + |u| ) ^2 M_u (v) } \right ) ^{1/2} \|\eta \|_{M_u},\;\;\eta \in \bhomu.
\]
\begin{pro}
There is a unique solution $F$ of the variational problem
\[
F \in \bthomu\;\mbox{ and }\; {\bf a}_u (F, \eta) = {\bf L } (\eta),\;\mbox{ for any } \eta \in \bhomu.
\]
For any $W \in \R^d$ we have $\psi _W (v) = F(v) \cdot W, v \in \R^d$. The vector field $F$ is left invariant by the family $\calTu$. 
\end{pro}
\begin{proof}
The bilinear for ${\bf a}_u$ is coercive on $\bthomu$, thanks to \eqref{Equ41}
\[
{\bf a}_u (\xi, \xi) \geq \frac{\min \{\sigma, \lambda _u \}}{2} \|\xi \|_{M_u } ^2,\;\;\mbox{ for any } \xi \in \bthomu. 
\]
By Lax-Milgram lemma, applied on the Hilbert space $\bthomu$, there is a unique vector field $F \in \bthomu$ such that 
\[
{\bf a}_u (F, \eta) = {\bf L } (\eta),\;\mbox{ for any } \eta \in \bthomu.
\]
Actually, the above equality holds true for any $\eta \in \bhomu$
\begin{align*}
{\bf a}_u (F, \eta) & = \sum _{i = 1} ^d a_u (F_i, \eta _i) = \sum _{i = 1} ^d a_u (F_i, \eta_i - (\eta_i,1)_{M_u} ) \\
& = {\bf L} ( \eta _1 - (\eta _1,1)_{M_u},...,\eta_d - (\eta _d ,1) _{M_u}) \\
& = \sum _{i =1 } ^d \intv{(v_i - u_i) [ \eta _i (v) - (\eta_i, 1)_{M_u}] M_u (v)} \\
& =  \sum _{i =1 } ^d \intv{(v_i - u_i) \eta _i (v)  M_u (v)} \\
& = {\bf L} (\eta).
\end{align*}
It remains to check that for any $W \in \R^d$, $v \mapsto F(v) \cdot W$ solves \eqref{Equ33}, on $\homu$. Observe that $F \cdot W \in \thomu$. Notice also that for any $\theta \in \homu$ we have $\theta W \in \bhomu$ and
\begin{align*}
a_u (F\cdot W, \theta) & = \sigma \intv{\;^t \partial _v F W \cdot \nabla _v \theta M_u (v)}\\
& = \sigma \intv{\partial _v F : \partial _v (\theta W) M_u (v)}\\
& = {\bf a}_u (F, \theta W) = {\bf L} (\theta W) \\
& = \intv{(v-u) \cdot W \theta(v) M_u (v) }\\
& = L(\theta).
\end{align*}
Thank to the uniqueness we obtain $\psi _W (v) = F(v) \cdot W, v \in \R^d, W \in \R^d$. Consider now $\calO \in \calTu$. We are done if we prove that $v \to \calO F (\;^t \calO v)$ solves the same problem as $F$. Clearly we have
\[
\intv{|\calO F( \;^t \calO v ) |^2 M_u (v)} = \intvp{|F(\vp)|^2 M_u (\vp) } < +\infty
\]
\begin{align*}
\intv{\partial [ \calO F (\;^t \calO \;\cdot)] : \partial [ \calO F (\;^t \calO \;\cdot )] M_u (v) } & = \intv{\partial F (\;^t \calO \;\cdot ) : \partial F (\;^t \calO \;\cdot ) M_u (v)} \\
& = \intv{\partial F ( \;^t \calO v ) \;^t \calO : \partial F ( \;^t \calO v) \;^t \calO M_u (v) } \\
& = \intvp{\partial F (\vp) : \partial F (\vp) M_u (\vp) } < +\infty. 
\end{align*}
and
\[
\intv{\calO F (\;^t \calO v ) M_u (v)} = \calO \intvp{F(\vp) M_u (\calO \vp) } = 0
\]
saying that $v \mapsto \calO F(\;^t \calO v)$ belongs to $\bthomu$. For any $\eta \in \bhomu$ we have $^t \calO \eta (\calO \; \cdot ) \in \bhomu$ and 
\begin{align*}
{\bf a}_u (\calO F (\;^t \calO \; \cdot ), \eta) & = \sigma \intv{\partial ( \calO F (\;^t \calO \; \cdot )) : \partial \eta M_u (v)} \\
& = \sigma \intv{\calO \partial F (\;^t \calO v ) \;^t \calO : \partial \eta M_u (v)}\\
& = \sigma \intvp{\partial F (\vp) : \;^t \calO ( \partial \eta) (\calO \vp) \calO M_u (\calO \vp) }\\
& = \sigma \intvp{\partial F (\vp) : \partial ( \;^t O \eta (\calO \; \cdot ) ) (\vp) M_u (\vp)} \\
& = \intvp{(\vp - u) \cdot \;^t \calO \eta (\calO \vp) M_u (\vp)} \\
& = \intv{(v-u) \cdot \eta (v) M_u (v)} = {\bf L}(\eta).
\end{align*}
\end{proof}
The vector field $F$ expresses in terms of two functions which are left invariant by the family $\calTu$.
\begin{pro}
There is a function $\psi$, which is left invariant by the family $\calTu$, such that 
\[
F(v) = \psi (v) \frac{v - (v \cdot \Omega)\Omega}{\sqrt{|v|^2 - (v \cdot \Omega)^2 }} + \psi _\Omega (v) \Omega,\;\;v \in \R ^d \setminus (\R \Omega).
\]
\end{pro}
\begin{proof}
Obviously we have $F = (F \cdot \Omega) \Omega + F^\prime = \psi _\Omega \Omega + F^\prime$, with $F^\prime = (I_d - \Omega \otimes \Omega) F$. The vector field $F^\prime$ is orthogonal to $\Omega$ and is left invariant by the family $\calTu$
\begin{align*}
F^\prime (\;^t \calO v) & = F (\;^t \calO v ) - ( F (\;^t \calO v )\cdot \Omega) \Omega  = \;^t \calO F(v) - ( \;^t \calO F(v) \cdot \Omega) \Omega \\
& = \;^t \calO (F(v) - (F(v) \cdot \Omega) \Omega ) = \;^t \calO F^\prime (v),\;\;v \in \R^d.
\end{align*}
We claim that $F^\prime (v)$ is parallel to the orthogonal projection of $v$ over $(\R \Omega) ^\perp$. Indeed, for any $v \in \R^d \setminus (\R \Omega)$, let us consider
\[
E(v) = \frac{(I_d - \Omega \otimes \Omega)v}{\sqrt{|v|^2 - (v \cdot \Omega)^2}}.
\]
When $d = 2$, since  $E(v)$ and $F^\prime (v)$ are both orthogonal to $\Omega$, there exists a function $\psi = \psi (v)$ such that 
\[
F^\prime(v) = \psi (v) E(v) = \psi (v) \frac{(I_2- \Omega \otimes \Omega)v}{\sqrt{|v|^2 - (\Omega \cdot v)^2}},\;\;v \in \R^2 \setminus ( \R \Omega) \, .
\]
If $d \geq 3$, let us denote by $^\bot E$, any unitary vector orthogonal to $E$ and $\Omega$. Introducing the orthogonal matrix $\calO = I_d - 2 \;^\bot E \otimes \;^\bot E \in \calTu$, we obtain $F^\prime  ( \;^t \calO \;\cdot) = \;^t \calO F^\prime $. Observe that
\[
0 = \;^\bot E \cdot E(v) = \;^\bot E\cdot \frac{v - (v \cdot \Omega)\Omega}{\sqrt{|v|^2 - (v \cdot \Omega) ^2}} = \frac{\;^\bot E\cdot v}{\sqrt{|v|^2 - (v \cdot \Omega) ^2}},\;\;\calO v = v
\]
and thus 
\[
F^\prime (v) = F^\prime  (\calO v ) = \calO F^\prime(v) = (I_d - 2 \;^\bot E \otimes \;^\bot E) F^\prime (v) = F^\prime (v) - 2 ( \;^\bot E \cdot F^\prime (v)) \;^\bot E
\]
from which it follows that $^\bot E \cdot F^\prime(v) = 0$, for any vector $^ \bot E$ orthogonal to $E$ and $\Omega$. Hence, there exists a function $\psi (v)$ such that 
\[
F^\prime (v) = \psi (v) E (v) = \psi (v) \frac{(I_d - \Omega \otimes \Omega)v}{\sqrt{|v|^2 - (v \cdot \Omega)^2}},\;\;v \in \R ^d \setminus (\R \Omega).
\]
It is easily seen that the function $\psi$ is left invariant by the family $\calTu$. Indeed, for any $\calO \in \calTu$ we have
\[
\psi (\;^t \calO v) = F^\prime (\;^t \calO v) \cdot E (\;^t \calO v) = \;^t \calO F^\prime (v) \cdot \;^t \calO E(v) = F^\prime (v) \cdot E(v) = \psi (v),\; v \in \R^d.
\]
Similarly, $\psi _\Omega$ is left invariant by the family $\calTu$
\[
\psi _\Omega (\;^t \calO v ) = F (\;^t \calO v ) \cdot \Omega = \;^t \calO F(v) \cdot \Omega = F(v) \cdot \calO \Omega = F(v) \cdot \Omega = \psi _\Omega (v),\;v \in \R^d, \;\calO \in \calTu.
\]
\end{proof}
The functions $\psi, \psi _\Omega$ will enter the fluid model satisfied by the macroscopic quantities $\rho, \Omega, |u|$. It is convenient to determine the elliptic partial differential equations satisfied by them.
\begin{pro}
There are two functions $\chi = \chi (c, r) : ]-1, 1[ \times ]0,+\infty[ \to \R$, $\chi _\Omega = \chi _\Omega (c,r) : ]-1,1[ \times ]0,+\infty[ \to \R$ such that $\psi (v) = \chi \left ( v \cdot \Omega/|v|, |v|\right ), \psi _\Omega (v) = \chi _\Omega \left ( v \cdot \Omega/|v|, |v|\right )$, $v \in \R^d \setminus (\R \Omega)$. The above functions satisfy
\begin{align}
\label{EquChi}
- & \sigma \partial _c \{ r ^{d-3} (1 - c^2) ^{\frac{d-1}{2}} e(c,r,|u|) \partial _c \chi \}  - \sigma \partial _r \{ r ^{d-1} (1 - c^2 ) ^{\frac{d-3}{2}} e(c, r, |u|) \partial _r \chi \} \\
& + \sigma (d-2) r ^{d-3} (1 - c^2) ^{\frac{d-5}{2}} e(c, r, |u|) \chi = r^d (1 - c^2) ^{\frac{d-2}{2}}e(c,r,|u|),\;(c,r) \in  ]-1,1[ \times ]0,+\infty[\nonumber
\end{align}
and
\begin{align}
\label{EquChiOme}
- & \sigma \partial _c \{ r ^{d-3} (1 - c^2) ^{\frac{d-1}{2}} e(c,r,|u|) \partial _c \chi _\Omega \} - \sigma \partial _r \{ r ^{d-1} (1 - c^2 ) ^{\frac{d-3}{2}} e(c, r, |u|) \partial _r \chi _\Omega \} \\
& = r^{d-1 } (r c - |u|)  (1 - c^2) ^{\frac{d-3}{2}}e(c,r,|u|),\;(c,r) \in  ]-1,1[ \times ]0,+\infty[\nonumber
\end{align}
where $e(c, r, l) = \exp \left ( - \frac{r^2}{2\sigma} + \frac{rcl}{\sigma} - \frac{V(r)}{\sigma} \right )$. 
\end{pro}
\begin{proof}
The function $\psi _\Omega = F \cdot \Omega$ satisfies
\begin{equation}
\label{Equ45}
\psi _\Omega \in \thomu\;\mbox{ and } \sigma \intv{\nabla _v \psi _\Omega \cdot \nabla _v \theta \;M_u (v) } = \intv{(v-u) \cdot \Omega \;\theta (v) M_u (v)},\;\theta \in \thomu. 
\end{equation}
By Remark \ref{Sym} we know that there is $\chi _\Omega = \chi _\Omega (c,r)$ such that $\psi _\Omega (v) = \chi _\Omega ( v\cdot \Omega /|v|, |v|), v \in \R ^d \setminus ( \R \Omega)$. As $\psi _\Omega$ belongs to $\thomu$, which is equivalent to
\[
\intv{\psi _\Omega (v) M_u (v)} = 0,\;\;\intv{|\nabla _v \psi _\Omega |^2 M_u (v) } < +\infty
\]
we are led to the Hilbert space
\begin{align*}
H_{\parallel, |u|} & = \{ h : ]-1,1[ \times ]0,+\infty[ \to \R,\;\; \intcr{r^{d-1}}{h(c,r) e(c,r,|u|) ( 1 - c^2) ^{\frac{d-3}{2}}} = 0,\\
& \intcr{r^{d-1}}{\left [ (\partial _c h )^2 \frac{1- c^2}{r^2} + (\partial _r h ) ^2\right ] e(c,r,|u|) ( 1 - c^2) ^{\frac{d-3}{2}}} < +\infty\}
\end{align*}
endowed with the scalar product
\[
(h, g)_{\parallel, |u|} =  \intcr{r^{d-1}}{\left [ \partial _c h  \partial _c g  \frac{1- c^2}{r^2} + \partial _r h \partial _r g\right ] e(c,r,|u|) ( 1 - c^2) ^{\frac{d-3}{2}}},\;h, g \in H_{\parallel, |u|}.
\]
Taking in \eqref{Equ45} $\theta (v) = h (   v\cdot \Omega /|v|, |v|)$, with $h \in H_{\parallel, |u|}$ (which means $\theta \in \thomu$), we obtain
\begin{align*}
\sigma & \intcr{r^{d-1}}{\left [\partial _c \chi _\Omega \partial _c h \frac{1- c^2}{r^2} + \partial _r \chi _\Omega \partial _r h  \right] e(c,r,|u|) ( 1 - c^2) ^{\frac{d-3}{2}}}\\
& = \intcr{r^{d-1}}{(r c - |u|) h(c,r) e(c,r,|u|) ( 1 - c^2) ^{\frac{d-3}{2}}}
\end{align*}
which implies \eqref{EquChiOme}. We focus now on the equation satisfied by $\psi$. Let us consider an orthonormal basis $\{E_1, ..., E_{d-1}\}$ of $(\R \Omega ) ^\perp$. By Remark \ref{Sym} we know that there is $\chi = \chi (c,r)$ such that $\psi (v) = \chi  (   v\cdot \Omega /|v|, |v|)$ and
\[
\psi _{E_i} (v) = F(v) \cdot E_i = \psi (v) \frac{v \cdot E_i}{\sqrt{|v|^2 - ( v \cdot \Omega )^2}} = \chi  (   v\cdot \Omega /|v|, |v|) \frac{v \cdot E_i}{\sqrt{|v|^2 - ( v \cdot \Omega )^2}}
\]
for $v \in \R^d \setminus (\R \Omega), i \in \{1, ..., d-1\}$. Let us consider $\psi _{E_i, h} (v) = h(   v\cdot \Omega /|v|, |v|) \frac{v \cdot E_i}{\sqrt{|v|^2 - ( v \cdot \Omega )^2}}$, where $h = h(c,r)$ is a function such that $\psi _{E_i, h } \in \homu$. Actually, once that $\psi _{E_i, h} \in \homu$, then $((\psi _{E_i, h}, 1))_{M_u} = \intv{h(   v\cdot \Omega /|v|, |v|) \frac{v \cdot E_i}{\sqrt{|v|^2 - ( v \cdot \Omega )^2}}M_u (v)} = 0$, saying that $\psi _{E_i, h} \in \thomu$. A straightforward computation shows that 
\begin{align*}
\nabla _v \psi _{E_i,h} & = \frac{v \cdot E_i}{\sqrt{|v|^2 - ( v \cdot \Omega )^2}}\left [\partial _c h \frac{I_d - \frac{v \otimes v}{|v|^2}}{|v|} \Omega + \partial _r h \frac{v}{|v|}    \right ]\\
& + h \left ( \frac{v \cdot \Omega}{|v|}, |v|\right ) \left [ I_d - \frac{(v - (v \cdot \Omega) \Omega ) \otimes v }{|v|^2 - (v \cdot \Omega)^2} \right ]\frac{E_i}{\sqrt{|v|^2 - ( v \cdot \Omega )^2}}
\end{align*}
and 
\begin{align*}
|\nabla _v \psi _{E_i,h}|^2 & = \frac{(v \cdot E_i )^2}{|v|^4} (\partial _c h )^2 + \frac{(v \cdot E_i )^2}{|v|^2 - (v \cdot \Omega) ^2} (\partial _r h )^2 \\
& + \frac{|v|^2 - (v \cdot \Omega) ^2 - (v \cdot E_i )^2 }{( |v|^2 - ( v \cdot \Omega )^2)^2} h ^ 2\left ( \frac{v \cdot \Omega}{|v|}, |v|\right ).
\end{align*}
The condition $\psi _{E_i,h} \in \homu$ writes
\[
\intv{(\psi _{E_i,h})^2 M_u (v)} < + \infty,\;\;\intv{|\nabla _v \psi _{E_i,h}|^2 M_u (v)} <  + \infty
\]
which is equivalent, thanks to the Poincar\'e inequality \eqref{Equ22} to
\begin{align*}
& \intv{|\nabla _v \psi _{E_i,h}|^2 M_u (v)} \\
& = \intv{ \left [ \frac{(v \cdot E_i )^2}{|v|^4} (\partial _c h )^2 + \frac{(v \cdot E_i )^2 (\partial _r h )^2}{|v|^2 - (v \cdot \Omega) ^2}  
 + \frac{|v|^2 - (v \cdot \Omega) ^2 - (v \cdot E_i )^2 }{( |v|^2 - ( v \cdot \Omega )^2)^2} h ^ 2 \right ] M_u (v)}< +  \infty
\end{align*}
and therefore to $h \in H_{\perp, |u|}$,  where we consider he Hilbert space 
\[
H_{\perp, |u|} = \{h : ]-1,1[ \times ]0,+\infty[ \to \R,\;\;\|h \|^2_{\perp, |u|} = (h,h)_{\perp, |u|} < +\infty\}
\]
endowed  with the scalar product
\begin{align*}
(g, h) _{\perp, |u|} & = \intcr{r^{d-1}}{\left [ \frac{1-c^2}{r^2} \partial _c g \partial _c h + \partial _r g \partial _r h + \frac{(d-2)gh}{r^2 (1- c^2)} \right ] \\
& e(c, r, |u|) (1 - c^2) ^{\frac{d-3}{2}}}, g, h, \in H_{\perp, |u|}. 
\end{align*}
Taking $\theta = \psi _{E_i, h} \in \thomu $ in \eqref{Equ33} leads to 
\[
\sigma \intv{\nabla _v \psi _{E_i} \cdot \nabla _v \psi _{E_i, h} M_u (v)} = \intv{\psi _{E_i, h} \;(v \cdot E_i) M_u (v) },\;\;h \in H_{\perp, |u|}
\]
or equivalently
\begin{align*}
\sigma & \intcr{r^{d-1}}{{\left [ \frac{1-c^2}{r^2} \partial _c \chi \partial _c h + \partial _r \chi \partial _r h + \frac{(d-2)\chi h}{r^2 (1- c^2)} \right ] e(c, r, |u|) (1 - c^2) ^{\frac{d-3}{2}}}} \\
& = \intcr{r^{d-1}}{r h(c,r) e(c, r, |u|) (1 - c^2) ^{\frac{d-2}{2}}},\;h \in H_{\perp, |u|}
\end{align*}
which implies \eqref{EquChi}.
\end{proof}

\section{The fluid model}
\label{MainTh}

The balances for the macroscopic quantities $\rho, u$ follow by using the elements in the kernel of $\calL _f ^\star$. 

\begin{proof} (of Theorem \ref{MainRes1})\\
The use of $\psi = 1 \in \ker \calL _f ^\star$ leads to \eqref{Equ51}. By Lemma \ref{Spectral}, we know that $(\R u ) ^\perp \subset \ker ( \calmu - \sigma I_d)$ and thus, for any $(t,x) \in \R_+ \times \R^d$, the vector field
\[
v \mapsto F^\prime (t,x,v) = \chi \left (\frac{v \cdot \Omega (t,x)}{|v|}, |v|   \right ) \frac{(I_d - \Omega (t,x) \otimes \Omega (t,x))v}{\sqrt{|v|^2 - (v \cdot \Omega (t,x))^2}} = \sumi \psi _{E_i} (v) E_i
\]
belongs to the kernel of $\calL _f ^\star$, implying that
\begin{equation*}
\intv{\partial _t \fz \;F^\prime (t,x,v)} + \intv{v \cdot \nabla _x \fz  \;F^\prime(t,x,v) } = 0,\;\;(t,x) \in \R_+ \times \R^d.
\end{equation*}
We have $\partial _t \fz = \partial _t \rho M_u + \rho \frac{M_u }{\sigma} (v - u) \cdot \partial _t u$ and we obtain
\begin{align*}
& \intv{\partial _t \fz \;F^\prime}  = \intv{\left ( \partial _t \rho + \frac{\rho }{\sigma} (v - u) \cdot \partial _t u \right ) \chi \left ( \frac{v \cdot \Omega}{|v|}, |v|\right ) \frac{v - ( v \cdot \Omega) \Omega}{\sqrt{|v|^2 - (v \cdot \Omega) ^2}}M_u (v)}\\
& = \partial _t \rho \intv{\chi \left ( \frac{v \cdot \Omega}{|v|}, |v|\right )\frac{v - ( v \cdot \Omega) \Omega}{\sqrt{|v|^2 - (v \cdot \Omega) ^2}}M_u (v) } \\
& + \frac{\rho}{\sigma} \intv{\chi \left ( \frac{v \cdot \Omega}{|v|}, |v|\right )M_u (v) \frac{[v - ( v \cdot \Omega) \Omega] \otimes [v - (v \cdot \Omega) \Omega + (v \cdot \Omega) \Omega - u]}{\sqrt{|v|^2 - (v \cdot \Omega) ^2}}}\;\partial _t u. \nonumber 
\end{align*}
It is easily seen (use the change of variable $v = (I_d - 2 E_i \otimes E_i) v^\prime, 1 \leq i \leq d-1$) that 
\[
\intv{\chi \frac{v - ( v \cdot \Omega) \Omega}{\sqrt{|v|^2 - (v \cdot \Omega) ^2}}M_u (v) } = \sumi \intv{\chi \frac{(v \cdot E_i)E_i}{\sqrt{|v|^2 - (v \cdot \Omega) ^2}}M_u (v) } = 0 ,
\]
\begin{align*}
\intv{\chi M_u  \frac{[v - ( v \cdot \Omega) \Omega] \otimes [(v \cdot \Omega) \Omega - u]}{\sqrt{|v|^2 - (v \cdot \Omega) ^2}}} = \sumi \intv{\chi M_u  \frac{(v \cdot E_i)E_i \otimes [(v \cdot \Omega) \Omega - u]}{\sqrt{|v|^2 - (v \cdot \Omega) ^2}}} = 0,
\end{align*}
and
\begin{align*}
\intv{\chi M_u  \frac{[v - ( v \cdot \Omega) \Omega] \otimes [v - (v \cdot \Omega) \Omega]}{\sqrt{|v|^2 - (v \cdot \Omega) ^2}}} & = \sum_{1 \leq i, j \leq d-1} \intv{\chi M_u  \frac{(v \cdot E_i) (v \cdot E_j)}{\sqrt{|v|^2 - (v \cdot \Omega) ^2}}}E_i \otimes E_j\\
& = \sumi \intv{\chi \frac{(v \cdot E_i )^2}{\sqrt{|v|^2 - (v \cdot \Omega) ^2}} M_u (v)} E_i \otimes E_i \\
& = \intv{\chi \frac{\sqrt{|v|^2 - (v \cdot \Omega) ^2}}{d-1} M_u (v) } (I_d - \Omega \otimes \Omega). 
\end{align*}
Therefore one gets 
\begin{equation}
\label{Equ53}
\intv{\partial _t f F^\prime (t,x,v) } = c_{\perp, 1} \frac{\rho}{\sigma} (I_d - \Omega \otimes \Omega)\partial _t u 
\end{equation}
with 
\[
c_{\perp, 1} = \intv{\chi \left (\frac{v \cdot \Omega }{|v|}, |v|   \right )\frac{\sqrt{|v|^2 - (v \cdot \Omega) ^2}}{d-1} M_u (v) }.
\]
Observe also that
\begin{align*}
v \cdot \nabla _x f & = (v \cdot \nabla _x \rho ) M_u + \frac{\rho}{\sigma} \partial _x u v \cdot (v-u) M_u \\
& = ( v \cdot \nabla _x \rho) M_u + \frac{\rho}{\sigma} \partial _x u v \cdot (v - (v \cdot \Omega) \Omega + (v \cdot \Omega)\Omega - u) M_u,
\end{align*}
and therefore 
\begin{align}
\label{Equ53Bis}
& \intv{(v \cdot \nabla _x f ) \;F^\prime}  = \intv{\chi \left ( \frac{v \cdot \Omega}{|v|}, |v|\right )M_u (v) \frac{(v - (v \cdot \Omega) \Omega)\otimes v}{\sqrt{|v|^2 - (v \cdot \Omega)^2}}} \;\nabla _x \rho \\
& + \frac{\rho}{\sigma} \intv{ \chi \left ( \frac{v \cdot \Omega}{|v|}, |v|\right )M_u (v) \frac{(v - (v \cdot \Omega)\Omega) \otimes (v - (v \cdot \Omega) \Omega + (v \cdot \Omega)\Omega - u)}{\sqrt{|v|^2 - (v \cdot \Omega)^2}} \; \partial _x u v}.\nonumber 
\end{align}
As before, using the change of variable $v = (I_d - 2E_i \otimes E_i ) v^\prime, 1 \leq i \leq d-1$, we have
\begin{align*}
\intv{\chi M_u (v)  \frac{(v - (v \cdot \Omega) \Omega)\otimes v}{\sqrt{|v|^2 - (v \cdot \Omega)^2}}} & = \intv{\chi M_u \frac{(v - (v \cdot \Omega) \Omega)\otimes (v - (v \cdot \Omega) \Omega)}{\sqrt{|v|^2 - (v \cdot \Omega)^2}}} \\
& + \intv{\chi M_u \frac{(v - (v \cdot \Omega) \Omega)\otimes (v \cdot \Omega) \Omega}{\sqrt{|v|^2 - (v \cdot \Omega)^2}}} \\
& = c_{\perp, 1} (I_d - \Omega \otimes \Omega). 
\end{align*}
For the second integral in the right hand side of \eqref{Equ53Bis}, by noticing that
\[
\intv{(v\cdot E_i) (v \cdot E_j) (v\cdot E_k) \chi \left ( \frac{v \cdot \Omega}{|v|}, |v|\right ) M_u (v) } = 0,\;\;i,j,k \in \{1, ..., d-1\},
\]
we obtain
\begin{align*}
& \intv{\chi M_u (v)  \frac{(v - (v \cdot \Omega) \Omega)\otimes ( v- (v \cdot \Omega) \Omega + (v \cdot \Omega )\Omega - u)}{\sqrt{|v|^2 - (v \cdot \Omega)^2}}\partial _x u v } \\
& = \intv{\chi M_u \frac{(v - (v \cdot \Omega) \Omega)\otimes (v - (v \cdot \Omega) \Omega)}{\sqrt{|v|^2 - (v \cdot \Omega)^2}}\partial _x u \Omega \;(v \cdot \Omega) } \\
& \quad + \intv{\chi M_u \frac{(v - (v \cdot \Omega) \Omega)\otimes ((v \cdot \Omega) \Omega - u)}{\sqrt{|v|^2 - (v \cdot \Omega)^2}}\partial _x u ( v - (v \cdot \Omega)\Omega) } \\
& = c_{\perp, 2} (I_d - \Omega \otimes \Omega) \partial _x u \Omega + \intv{\chi M_u \frac{(v - (v \cdot \Omega) \Omega)\otimes (v - (v \cdot \Omega) \Omega)}{\sqrt{|v|^2 - (v \cdot \Omega)^2}}\;^t \partial _x u [(v \cdot \Omega)\Omega - u] } \\
& = c_{\perp, 2} (I_d - \Omega \otimes \Omega) (\partial _x u + \;^t \partial _x u ) \Omega - c_{\perp, 1} (I_d - \Omega \otimes \Omega) \; (u\cdot \partial _x) u,
\end{align*}
where
\[
c_{\perp, 2} = \intv{(v\cdot \Omega) \chi \frac{\sqrt{|v|^2 - (v \cdot \Omega)^2}}{d-1} M_u (v)}.
\]
Therefore we deduce
\begin{align}
\label{Equ54}
\intv{(v \cdot \nabla _x f) F^\prime (t,x,v)} & = c_{\perp, 1} (I_d - \Omega \otimes \Omega) \nabla _x \rho + \frac{\rho}{\sigma} c_{\perp, 2} (I_d - \Omega \otimes \Omega) (\partial _x u + \;^t \partial _x u ) \Omega \nonumber \\
& \quad -  \frac{\rho}{\sigma} c_{\perp, 1}(I_d - \Omega \otimes \Omega)\; (u\cdot \partial _x) u 
\end{align}
and finally \eqref{Equ53}, \eqref{Equ54} yield
\begin{equation}
\label{Equ55}
(I_d - \Omega \otimes \Omega) \partial _t u + \sigma (I_d - \Omega \otimes \Omega)\frac{\nabla _x \rho}{\rho} + c_\perp (I_d - \Omega \otimes \Omega)(u\cdot \partial _x) u + (c_\perp - 1) (I_d - \Omega \otimes \Omega) \nabla _x \frac{|u|^2}{2} = 0
\end{equation}
where
\begin{align*}
c_\perp & = \frac{c_{\perp, 2}}{|u| \; c_{\perp, 1}} = \frac{\intv{(v \cdot \Omega) \chi\left ( \frac{v \cdot \Omega}{|v|}, |v|\right ) \sqrt{|v|^2 - (v \cdot \Omega) ^2} M_u (v)}}{|u| \;\intv{ \chi \left ( \frac{v \cdot \Omega}{|v|}, |v|\right ) \sqrt{|v|^2 - (v \cdot \Omega) ^2} M_u (v)}}\\
& = \frac{\int_{\R_+} r^{d+1} \intth{\cos \theta \chi (\cos \theta, r) e(\cos \theta, r, l(\sigma)) \sin ^{d-1} \theta }\md r}{l(\sigma) \int_{\R_+} r^{d} \intth{ \chi (\cos \theta, r) e(\cos \theta, r, l(\sigma)) \sin ^{d-1} \theta }\md r}.
\end{align*}
Recall that $|u| = l(\sigma)$ and therefore we have $u \cdot \partial _t u = \frac{1}{2} \partial _t |u|^2 = 0$, $(u\cdot \partial _x) u = \frac{1}{2} \nabla _x |u|^2 = 0$, implying that 
\[
\Omega \cdot \partial _t u = 0,\;\;^t \partial _x u \Omega = 0,\;\;\Omega \cdot \partial _x u \Omega = 0.
\]
The equation \eqref{Equ55} becomes
\[
\partial _t \Omega + l(\sigma) c_\perp ( \Omega \cdot \nabla _x ) \Omega + \frac{\sigma}{l(\sigma)} (I_d - \Omega \otimes \Omega) \frac{\nabla _x \rho}{\rho} = 0.
\]
We have to check that $c_{\perp, 1} \neq 0$. This comes by using the elliptic equations satisfied by $\psi _{E_i}$, that is 
\[
- \sigma \Divv ( M_u \nabla _v \psi _{E_i}) = (v \cdot E_i) M_u (v),\;\;v \in \R^d,\;\;i \in \{1, ..., d-1\}.
\]
Indeed, we have
\begin{align*}
c_{\perp, 1} & = \intv{\chi \frac{(v \cdot E_i )^2}{\sqrt{|v|^2 - (v \cdot \Omega) ^2}}M_u (v)} = \intv{(F(v) \cdot E_i) (v \cdot E_i) M_u (v)} \\
& = \intv{\psi _{E_i} (v) (v \cdot E_i) M_u (v)} = \sigma \intv{|\nabla _v \psi _{E_i} |^2 M_u (v)} >0.
\end{align*}
\end{proof}
Other potentials $v \mapsto V(|v|)$ can be handled as well. For example, let us assume that 
there is $\sigma >0$, $ 0 \leq l_1 (\sigma) < l_2 (\sigma)\leq +\infty$ such that the function $l \mapsto Z(\sigma, l)$ is stricly increasing on $[0,l_1(\sigma)]$, constant on $[l_1(\sigma), l_2(\sigma)[$, and strictly decreasing on $[l_2(\sigma), +\infty[$. In that case, for any $l \in [l_1(\sigma), l_2(\sigma)[$ we have $\partial _{ll} ^2 Z(\sigma, l) = 0$ and by Lemma \ref{Spectral} we deduce that $\calmu = \sigma I_d$,  saying that $\ker ( \calmu - \sigma I_d) = \R ^d$. Using the function $\psi _\Omega$, we obtain a balance for $|u|$ as well. 
\begin{proof} (of Theorem \ref{MainRes2})\\
In this case $\psi _\Omega$ belongs to $\ker \calL _f ^\star$, and therefore we also have the balance
\[
\intv{\partial _t \fz \psi _{\Omega (t,x)}(v)} + \intv{(v \cdot \nabla _x \fz ) \psi _{\Omega (t,x)} (v) } = \intv{\calL _{\fz(t,x,\cdot)} (\fo ) \psi _{\Omega (t,x)}} = 0.
\]
As before, using also $\intv{\psi _\Omega (v) M_u (v)} = 0$, we write
\begin{align}
\label{Equ61}
\intv{\partial _t \fz \psi _\Omega} & = \intv{\left [ \partial _t \rho M_u (v)+ \frac{\rho}{\sigma} M_u (v)(v-u) \cdot \partial _t u  \right ]\psi _\Omega (v)}\\
& = \left ( \partial _t \rho - \frac{\rho}{\sigma} u \cdot \partial _t u \right ) \intv{\psi _\Omega (v) M_u (v)}  \nonumber \\
& + \frac{\rho}{\sigma} \intv{\chi _\Omega M_u (v) [v - (v \cdot \Omega) \Omega + (v \cdot \Omega) \Omega]} \cdot \partial _t u \nonumber \\
& = \frac{\rho}{\sigma} c_{\parallel, 1}  \Omega \cdot \partial _t u \nonumber
\end{align}
where
\[
 c_{\parallel, 1} = \intv{(v \cdot \Omega) \psi _\Omega (v) M_u (v)} = \intv{(v - u) \cdot \Omega \;\psi _\Omega  M_u } = \sigma \intv{|\nabla _v \psi _\Omega |^2 M_u (v)} >0.
\]
Similarly, observe that 
\begin{align}
\label{Equ62}
& \intv{(v \cdot \nabla _x f ) \psi _\Omega }  = \intv{\left [v \cdot \nabla _x \rho + \frac{\rho}{\sigma} \partial _x u v \cdot (v-u) \right ] M_u (v)\psi _\Omega (v)}\\
& \intv{(v \cdot \Omega) \psi _\Omega (v) M_u (v) } \;( \Omega \cdot \nabla _x \rho) + \frac{\rho}{\sigma} \intv{\psi _\Omega (v) M_u (v) (v-u) \otimes v } : \partial _x u \nonumber \\
& = c_{\parallel, 1} \Omega \cdot \nabla _x \rho + \frac{\rho}{\sigma} \intv{\psi _\Omega M_u  \{ [ v - (v \cdot \Omega) \Omega ] \otimes [v - (v \cdot \Omega) \Omega ] + (v \cdot \Omega)^2 \Omega \otimes \Omega\}}: \partial _x u \nonumber \\
& - \frac{\rho}{\sigma} \intv{\psi _\Omega (v) M_u (v) v } \cdot \;^t \partial _x u u = c_{\parallel, 1} \Omega \cdot \nabla _x \rho \nonumber \\
& + \frac{\rho}{\sigma} \left \{ \intv{\psi _\Omega  M_u  \frac{|v|^2 - (v \cdot \Omega)^2}{d-1} } \;(I_d - \Omega \otimes \Omega)  + \intv{(v \cdot \Omega)^2 \psi _\Omega  M_u  } \;\Omega \otimes \Omega \right \}: \partial _x u \nonumber \\
& - \frac{\rho}{\sigma} c_{\parallel, 1} \Omega \cdot \;^t \partial _x u u \nonumber \\
& = c_{\parallel, 1} \Omega \cdot \nabla _x \rho + \frac{\rho}{\sigma} ( 2 c_{\parallel, 2} - |u| c_{\parallel, 1} ) \Omega \otimes \Omega : \partial _x u + \frac{\rho}{\sigma} c_{\parallel, 3} (I_d - \Omega \otimes \Omega) : \partial _x u  \nonumber \\
& = c_{\parallel, 1} \Omega \cdot \nabla _x \rho + \frac{\rho}{\sigma} \frac{c_{\parallel, 2}}{|u|} ( \Omega \cdot \partial _x u u ) + \frac{\rho}{\sigma} \left ( \frac{c_{\parallel, 2}}{|u|} - c_{\parallel, 1} \right ) \left ( \Omega \cdot \nabla _x \frac{|u|^2}{2} \right ) + \frac{\rho}{\sigma} c_{\parallel, 3} |u|\Divx \Omega \nonumber
\end{align}
where
\[
c_{\parallel, 2} = \intv{\frac{(v \cdot \Omega) ^2}{2} \psi _\Omega (v) M_u (v)},\;\;c_{\parallel, 3} = \intv{\frac{|v|^2 - (v \cdot \Omega) ^2}{d-1} \psi _\Omega (v) M_u (v)}.
\]
In the above computations we have used the identity $(I_d - \Omega \otimes \Omega) : \partial _x u = |u| \Divx \Omega$. Combining \eqref{Equ61}, \eqref{Equ62} leads to
\begin{equation}
\label{Equ63}
\Omega \cdot \partial _t u + \sigma \Omega \cdot \frac{\nabla _x \rho}{\rho} + c_\parallel ( \Omega \cdot (u\cdot \partial _x) u ) + ( c_\parallel - 1) \left ( \Omega \cdot \nabla _x \frac{|u|^2}{2} \right ) + c_\parallel ^\prime |u|^2 \Divx \Omega = 0
\end{equation}
where $c_\parallel = \frac{c_{\parallel, 2}}{|u| c_{\parallel, 1}}, \;c_\parallel ^\prime= \frac{c_{\parallel, 3}}{|u| c_{\parallel, 1}}$. Finally we deduce from \eqref{Equ55}, \eqref{Equ63} the balance for the mean velocity $u$
\begin{align}
\partial _t u & + c_\perp (I_d - \Omega \otimes \Omega) \partial _x u u + c_\parallel ( \Omega \otimes \Omega) (u\cdot \partial _x) u + (c_\perp - 1) (I_d - \Omega \otimes \Omega) \nabla _x \frac{|u|^2}{2} \nonumber \\
& + (c _\parallel - 1) (\Omega \otimes \Omega )\nabla _x \frac{|u|^2}{2}  + \sigma \frac{\nabla _x \rho }{\rho} + c^\prime _\parallel \Divx \Omega |u| u = 0.\nonumber
\end{align}
\end{proof}
\begin{remark}
\label{Maxwellian}
When $V = 0$, the equilibria are Maxwellians parametrized by $\rho \in \R_+$ and $u \in \R^d$
\[
M_u (v) = \frac{\rho}{(2\pi \sigma) ^{d/2}} \exp \left ( - \frac{|v-u|^2}{2\sigma} \right ),\;\;v \in \R^d.
\]
In that case the function $l \to Z(\sigma, l)$ is constant
\[
Z(\sigma, l) = \intv{\exp \left ( - \frac{|v-u|^2}{2\sigma} \right )} = (2\pi \sigma ) ^{d/2},\;\;l \in \R_+ ^\star.
\]
It is easily seen that the solution of
\[
- \sigma \Divv \{ M_u \partial _v F\} = (v-u) M_u (v),\;\;v \in \R^d,\;\;\intv{M_u (v) F(v) } = 0
\]
is $F(v) = v-u, v \in \R^d$, which belongs to $\bthomu$, and therefore the functions $\psi, \psi _\Omega$ such that 
\[
F(v) = \psi (v) \frac{v - (v \cdot \Omega) \Omega}{\sqrt{|v|^2 - (v \cdot \Omega)^2}} + \psi _\Omega (v) \Omega,\;\;v \in \R^d \setminus ( \R \Omega)
\]
are given by
\[
\psi (v) = (v-u) \cdot \frac{v - (v \cdot \Omega) \Omega}{\sqrt{|v|^2 - (v \cdot \Omega)^2}} = \sqrt{|v|^2 - (v \cdot \Omega)^2},\;\;\psi _\Omega (v) = (v - u) \cdot \Omega,\;\;v \in \R^d
\]
and $\psi _{E_i}(v) = F(v) \cdot E_i = (v \cdot E_i), v \in \R^d, 1\leq i \leq d-1$.
\end{remark}
By straightforward computations we obtain
\[
c_{\perp, 1} = \intv{\frac{|v|^2 - (v \cdot \Omega) ^2}{d-1} M_u } = \intv{(v\cdot E_1) ^2 M_u } = - \sigma \intv{(v \cdot E_1) (\nabla _v M_u \cdot E_1) } = \sigma
\]
\begin{align*}
c_{\perp, 2} & = \intv{(v \cdot \Omega) \frac{|v|^2 - (v \cdot \Omega) ^2}{d-1} M_u } = \intv{(v \cdot \Omega) (v\cdot E_1) ^2 M_u } \\
& =- \sigma \intv{(v \cdot \Omega) ( v \cdot E_1) \Divv (M_u E_1) } = \sigma \intv{M_u (v)E_1 \cdot [ ( v \cdot E_1) \Omega + (v \cdot \Omega) E_1]} \\
& = \sigma |u|.
\end{align*}
\[
c_\perp = \frac{c_{\perp, 2}}{|u|c_{\perp, 1}} = 1
\]
\[
c_{\parallel, 1} = \sigma \intv{|\nabla _v \psi _\Omega|^2 M_u (v)} =  \sigma
\]
\[
c_{\parallel, 2} = \intv{\frac{(v \cdot \Omega)^2}{2} \psi _\Omega  M_u} =  - \sigma \intv{\frac{(v \cdot \Omega)^2}{2}\Divv ( M_u \Omega) } = \sigma \intv{(v \cdot \Omega) M_u } = \sigma |u|
\]
\[
c_\parallel = \frac{c_{\parallel, 2}}{|u|c_{\parallel, 1}} = 1
\]
\begin{align*}
c_{\parallel, 3} & = \intv{\frac{|v|^2 - (v \cdot \Omega) ^2}{d-1} \psi _\Omega M_u } = \intv{(v \cdot E_1) ^2 \psi _\Omega M_u } = - \sigma \intv{(v \cdot E_1)^2 \Divv ( M_u \Omega) } \\
& = 2\sigma \intv{(v\cdot E_1) (E_1 \cdot \Omega) M_u } = 0.
\end{align*}
In this case \eqref{Equ57}, \eqref{Equ58} are the Euler equations, as expected when taking the limit $\eps \searrow 0$ in the Fokker-Planck equations
\[
\partial _t \fe + v \cdot \nabla _x \fe = \frac{1}{\eps} \Divv \{ \sigma \nabla _v \fe + \fe ( v - u[\fe])\},\;\;(t,x,v) \in \R_+ \times \R^d \times \R^d 
\]
that is
\[
\partial _t \rho + \Divx (\rho u) = 0,\;\;\partial _t u + \partial _x u u + \sigma \frac{\nabla _x \rho}{\rho} = 0,\;\;(t,x) \in \R_+ \times \R^d. 
\]

\section{Examples}
\label{Example}

We analyze now the potentials $v \mapsto \vab = \beta \frac{|v|^4}{4} - \alpha \frac{|v|^2}{2}$. Clearly the hypothesis \eqref{Equ11} is satisfied, and thus the function $Z(\sigma, |u|) = \intv{\exp \left ( - \frac{|v-u|^2}{2\sigma} - \frac{\vab{}}{\sigma} \right )}$ is well defined. As seen in Section \ref{PropEqui}, the sign of $\partial _l Z(\sigma, l)$, for small $\sigma >0$, depends on the sign of $V^\prime _{\alpha, \beta}$. The potential $V_{\alpha, \beta}$ satisfy \eqref{Equ15} with $r _0 = \sqrt{\alpha/\beta}$
\[
V^\prime _{\alpha, \beta} (r) = r ( \beta r^2 - \alpha )<0\;\mbox{ for } \;0 < r < \sqrt{\alpha/\beta}\;\mbox{ and } V^\prime _{\alpha, \beta } (r) >0 \;\mbox{ for any } r >  \sqrt{\alpha/\beta}.
\]
One can check that these potentials belong to the family $\mathcal{V}$, see \cite{Li19}. We include an example $V_{1,1} (|v|) = \frac{|v|^4}{4} - \frac{|v|^2}{2}$ for the sake of completeness. In this case the critical diffusion can be computed explicitly.
\begin{pro}
Consider the potential $v \mapsto V_{1,1}(|v|) = \frac{|v|^4}{4} - \frac{|v|^2}{2}$. The critical diffusion $\sigma _0 $ writes
\[
\sigma _0 ^{1/2} = \frac{1}{d} \frac{\int_{\R_+} \exp( - z^4/4) z^{d+1}\;\md z}{\int_{\R_+} \exp( - z^4/4) z^{d-1}\;\md z},\;\;d \geq 2.
\]
In particular, for $d = 2$ we have $\sigma _0 = 1/\pi$. 
\end{pro}
\begin{proof}
We have 
\[
\Phiu (v) = \frac{|v-u|^2}{2} + V_{1,1} (|v|) = \frac{|v|^4}{4} - v \cdot u + \frac{|u|^2}{2}
\]
and therefore
\begin{align*}
Z(\sigma, l) & = \intv{\exp\left ( - \frac{\Phiu (v)}{\sigma} \right ) }= |\mathbb{S} ^{d-2}| \exp \left ( - \frac{l^2}{2\sigma} \right ) \\
& \int _{\R_+} \exp \left ( - \frac{r^4}{4\sigma} \right ) r^{d-1} \intth{\exp \left ( \frac{r l \cos \theta}{\sigma} \right ) \sin ^{d-2} \theta }\md r.
\end{align*}
Taking the second derivative with respect to $l$ one gets cf. Remark \ref{Modulus}
\begin{align*}
\partial _{ll} ^2 Z(\sigma, l)  & = \intv{\exp \left ( - \frac{|v|^4}{4\sigma} + \frac{v \cdot u }{\sigma} - \frac{l^2}{2\sigma} \right ) \frac{(v \cdot \Omega - l) ^2 - \sigma}{\sigma ^2} } = |\mathbb{S} ^{d-2}| \exp \left ( - \frac{l^2}{2\sigma} \right ) \\
& \int _{\R_+} \exp \left ( - \frac{r^4}{4\sigma} \right ) r^{d-1} \intth{\exp \left ( \frac{r l \cos \theta}{\sigma} \right ) \frac{(r  \cos \theta - l)^2 - \sigma }{\sigma ^2}\sin ^{d-2} \theta }\md r
\end{align*}
and therefore
\begin{align*}
\partial _{ll} ^2 Z(\sigma, 0) & = \frac{|\mathbb{S} ^{d-2}|}{\sigma ^2} \int _{\R_+} \exp \left ( - \frac{r^4}{4\sigma} \right ) r^{d+1} \;\md r \intth{\cos ^2 \theta  \sin ^{d-2} \theta }   \\
& - \frac{|\mathbb{S} ^{d-2}|}{\sigma}\int _{\R_+} \exp \left ( - \frac{r^4}{4\sigma} \right ) r^{d-1} \;\md r \intth{ \theta  \sin ^{d-2} \theta }. 
\end{align*}
It is easily seen that 
\begin{align*}
\intth{\cos ^2 \theta  \sin ^{d-2} \theta } & = \intth{\sin ^{d-2} \theta} + \intth{\cos ^\prime \theta \sin ^{d-1} \theta } \\
& = \intth{\sin ^{d-2}\theta} - (d-1) \intth{\cos ^2 \theta \sin ^{d-2} \theta }
\end{align*}
and thus 
\[
\intth{\cos ^2 \theta  \sin ^{d-2} \theta } = \frac{1}{d} \intth{\sin ^{d-2} \theta},\;\;d \geq 2. 
\]
We obtain the following expression for the second derivative $\partial _{ll} ^2 Z(\sigma, 0)$
\[
\frac{\partial _{ll} ^2 Z(\sigma, 0)}{|\mathbb{S} ^{d-2}|} =  \frac{\intth{\sin ^{d-2} \theta }}{\sigma ^2}  \left \{ 
\frac{1}{d} \int _{\R_+} \exp \left ( - \frac{r^4}{4\sigma} \right ) r^{d+1} \;\md  r
- \sigma \int _{\R_+} \exp \left ( - \frac{r^4}{4\sigma} \right ) r^{d-1} \;\md r  \right \}.
\]
Using the change of variable $r = \sigma ^{1/4}z$, we have
\[
\int _{\R_+} \exp \left ( - \frac{r^4}{4\sigma} \right ) r^{d+1} \;\md  r = \int _{\R_+} \exp \left ( - \frac{z^4}{4} \right ) z^{d+1} \;\md  z \;\sigma ^{\frac{d+2}{4}}
\]
\[
\int _{\R_+} \exp \left ( - \frac{r^4}{4\sigma} \right ) r^{d-1} \;\md  r = \int _{\R_+} \exp \left ( - \frac{z^4}{4} \right ) z^{d-1} \;\md  z \;\sigma ^{\frac{d}{4}}
\]
and thus $\partial _{ll} ^2 Z(\sigma, 0) >0$ iff 
\[
\sigma ^{1/2} < \frac{1}{d} \frac{\int _{\R_+} \exp \left ( - \frac{z^4}{4} \right ) z^{d+1} \;\md  z}{\int _{\R_+} \exp \left ( - \frac{z^4}{4} \right ) z^{d-1} \;\md  z}.
\]
The critical diffusion $\sigma _0$ is, cf. Proposition \ref{CriticalDiffusion}
\[
\sigma _0 ^{1/2} = \frac{1}{d} \frac{\int_{\R_+} \exp( - z^4/4) z^{d+1}\;\md z}{\int_{\R_+} \exp( - z^4/4) z^{d-1}\;\md z},\;\;d \geq 2.
\]
In particular, when $d = 2$, we obtain
\[
\int_{\R_+} \exp( - z^4/4) z^{3}\;\md z = \int_{\R_+} \exp( - z^4/4)\; \md \frac{z^4}{4} = \int _{\R_+} \exp (-s) \;\md s = 1
\]
and
\begin{align*}
\int_{\R_+} \exp( - z^4/4) z\;\md z =\int_{\R_+} \exp( - z^4/4)\; \md \frac{z^2}{2} = \int _{\R_+} \exp (-s^2) \;\md s = \frac{\sqrt{\pi}}{2}
\end{align*}
implying that $\sigma _0 = 1/\pi$.
\end{proof}



\subsection*{Acknowledgments}
This work was initiated during the visit of MB to Imperial College London with an ICL-CNRS Fellowship. MB was supported by the French Federation for Magnetic Fusion Studies (FR-FCM) and the Eurofusion consortium, and has received funding from the Euratom research and training programme 2014-2018 and 2019-2020 under grant agreement No 633053. The views and opinions expressed herein do not necessarily reflect those of the European Commission.
MB was supported by the A*MIDEX project (no ANR-11-IDEX-0001-02) funded by the Investissements d'Avenir French Government program, managed by the French National Research Agency (ANR).
JAC was partially supported by the EPSRC grant number EP/P031587/1.


\end{document}